\documentclass[10pt]{amsart}

\usepackage[T1]{fontenc}
\usepackage[utf8]{inputenc}
\usepackage[english]{babel}
\usepackage{lmodern}

\usepackage[a4paper,margin=1.15in]{geometry}

\usepackage{xcolor}
\definecolor{DarkRed}{RGB}{173,0,0}
\definecolor{LightRed}{RGB}{201,0,0}

\usepackage{amsmath,amssymb,amsfonts,amsthm,mathtools}

\numberwithin{equation}{section}

\usepackage{tikz}
\usetikzlibrary{arrows,cd}
\DeclareMathOperator{\Hom}{Hom}
\DeclareMathOperator{\rk}{rk}

\DeclareMathOperator{\NE}{\overline{\mathrm{NE}}}
\DeclareMathOperator{\Sym}{Sym}

\DeclareMathOperator{\Hol}{Hol}

\DeclareMathOperator{\locus}{locus}
\DeclareMathOperator{\Supp}{Supp}

\DeclareMathOperator{\Aut}{Aut}

\newcommand{\Proj}{\operatorname{Proj}}
\newcommand{\cG}{\mathcal{G}}
\DeclareMathOperator{\Ext}{Ext}
\newcommand{\Z}{\mathbb{Z}}
\newcommand{\C}{\mathbb{C}}
\newcommand{\A}{\mathbb{A}}
\newcommand{\PP}{\mathbb{P}}
\newcommand{\OO}{\mathcal{O}}
\newcommand{\Om}{\Omega}

\theoremstyle{plain}
\newtheorem{theorem}[equation]{Theorem}
\newtheorem{lemma}[equation]{Lemma}
\newtheorem{proposition}[equation]{Proposition}
\newtheorem{corollary}[equation]{Corollary}

\theoremstyle{definition}
\newtheorem{definition}[equation]{Definition}

\theoremstyle{remark}
\newtheorem{remark}[equation]{Remark}
\newtheorem{example}[equation]{Example}

\usepackage[
  colorlinks=true,
  linkcolor=DarkRed,
  urlcolor=LightRed,
  citecolor=LightRed
]{hyperref}

\title{Log-Conformal Projective Manifolds}
\author[Maur\'icio Corr\^ea]{Maur\'icio Corr\^ea}
\address[Maur\'icio Corr\^ea]{Dipartimento di Matematica, Universit\`a degli Studi di Bari, Via E.~Orabona 4, I-70125 Bari, Italy}
\email{mauricio.barros@uniba.it}

\author[Alex Massarenti]{Alex Massarenti}
\address[Alex Massarenti]{Dipartimento di Matematica e Informatica, Universit\`a di Ferrara, Via Machiavelli 30, 44121 Ferrara, Italy}
\email{msslxa@unife.it}

\date{\today}


\begin{document}

\begin{abstract}
Let $(X,\Delta)$ be a smooth complex projective simple normal crossing pair of dimension $n\ge 3$ endowed with an everywhere nondegenerate logarithmic conformal tensor. If $K_X+\Delta$ is not nef, then exactly one of the following occurs: $\Delta=\varnothing$ and $X\simeq Q^n$; $X\simeq\PP^n$ and $\Delta$ is a hyperplane; or $n=2m$ and $(X,\Delta)$ admits a $(K_X+\Delta)$-negative elementary contraction $\phi:X\to Y$ generated by minimal rational curves of conformal degree one. In the third case, if $\dim Y>0$, then for some $1\le s\le m$ one has $\dim Y=m-s+1$ and the geometric generic fiber is $(\PP^{m+s-1},H_1+\cdots+H_s)$, where the $H_i$ are hyperplanes in general position; over a dense open subset, $\phi$ is a projective bundle and the horizontal boundary is the sum of $s$ relative hyperplanes. The case $s=1$ yields a rational maximal isotropic fibration with generic fiber $(\PP^m,H)$; if $\dim Y=0$, then $\rho(X)=1$ and $-(K_X+\Delta)$ is ample. If $K_X+\Delta\equiv0$, then, under a Bochner extension principle, a restricted-holonomy condition, and trivial monodromy of the induced flat connection on $X\setminus\Delta$, a logarithmic conformal tensor with trivial conformal line bundle forces $X\setminus\Delta$ to be semi-abelian and $(X,\Delta)$ to be its toroidal compactification.
\end{abstract}

\maketitle
\setcounter{tocdepth}{1}
\tableofcontents

\section{Introduction}

Holomorphic conformal geometry lies at the intersection of locally homogeneous geometric structures, uniformization by Hermitian symmetric domains, the birational geometry of projective manifolds, and the twistorial study of null directions and their incidence relations. In the foundational work of Kobayashi and Ochiai \cite{KO82}, holomorphic conformal structures are treated as geometric structures modeled on the hyperquadric, leading to strong rigidity and uniformization results. From a different direction, the Penrose program \cite{Penrose67} and its mathematical development by Atiyah--Hitchin--Singer \cite{AHS78} and LeBrun \cite{LeBrun83} emphasize that conformal geometry is governed not by a distinguished metric, but by its null cones and null curves. In LeBrun's complex-geometric reformulation \cite{LeBrun83}, the space of null geodesics becomes an incidence space closely related to contact geometry, from which the underlying manifold may, in principle, be recovered. In the projective setting, Ye \cite{Ye94} gave an algebro-geometric formulation of this null geometry inside $\PP(TX)$: the divisor of null directions, the natural contact structure, and suitable families of rational curves provide the language in which conformal geometry can be studied birationally. In this setting, spaces of rational curves and their deformation theory are the natural moduli spaces.

The logarithmic setting extends these ideas. Already in dimension two, Kobayashi and Naruki \cite{KN88} showed that everywhere regular conformal structures can be replaced by generalized conformal tensors with controlled logarithmic behaviour along a boundary divisor, thereby including orbifold singularities, Hilbert modular cusps, and ramified coverings. More recently, Langer \cite{Langer24} showed that logarithmic pairs also carry natural contact-geometric structures, suggesting a logarithmic extension of the classical conformal, contact, and twistorial pictures. The objects considered here are smooth complex projective simple normal crossing pairs $(X,\Delta)$ endowed with an everywhere nondegenerate logarithmic conformal tensor on $T_X(-\log \Delta)$. The central question is how much of the rigidity of the classical theory persists in the presence of the boundary and which logarithmic phenomena may occur. In the numerically trivial case, this question is also related to the recent analysis of holomorphic tensors on log Calabi--Yau manifolds by Collins and Guenancia \cite{CollinsGuenancia25}.

From the birational standpoint, the logarithmic setting is equally natural. Since singular models are unavoidable in the minimal model program \cite{KM98}, it is natural to consider a conformal structure on a log resolution and ask whether it descends, after contracting the boundary in codimension at least two, to a reflexive conformal structure on the singular model. Thus logarithmic conformal geometry relates smooth pairs with boundary to compact singular models. The final part of the article treats boundary contractions and the resulting reflexive conformal structures.

Let $(X,\Delta)$ be a smooth complex projective simple normal crossing pair of dimension $n\ge 3$, and let $g\in H^0\bigl(X,\Sym^2\Omega_X^1(\log\Delta)\otimes N\bigr)$
be a fiberwise nondegenerate logarithmic conformal tensor, where $N$ is a line bundle on $X$.
The first result concerns the non-nef case. Besides the smooth quadric and projective space with hyperplane boundary, there is an even-dimensional alternative whose positive-dimensional extremal contractions have projective-space fibers with a simple normal crossing hyperplane arrangement. The one-hyperplane case is a rational maximal isotropic fibration.

\begin{theorem}
Let $(X,\Delta)$ be a smooth complex projective simple normal crossing pair of dimension $n\ge 3$ carrying an everywhere nondegenerate log-conformal tensor $(N,g)$. Assume that $K_X+\Delta$ is not nef. Then one of the following mutually exclusive alternatives occurs:
\begin{enumerate}
    \item $\Delta=\varnothing$ and $X\simeq Q^n$;
    \item $X\simeq \PP^n$ and $\Delta$ is a hyperplane;
    \item $n=2m$ is even and there is a $(K_X+\Delta)$-negative elementary contraction $\phi_R:X\to Y$, generated by a minimal rational curve $C$ with $N\cdot C=1$, for which one of the following holds:
    \begin{enumerate}
        \item $\dim Y>0$, and for some $1\le s\le m$, $\dim Y=m-s+1$ and the general fiber is $(F,\Delta|_F)\simeq (\PP^{m+s-1},H_1+\cdots+H_s),$
        where the $H_i$ are hyperplanes in general position and $N|_F\simeq\OO_{\PP^{m+s-1}}(1)$. After shrinking the base, $\phi_R$ is a projective bundle and the horizontal boundary is the sum of $s$ relative hyperplanes. If $s=1$, this is the rational maximal isotropic fibration described by an exact sequence $0\to\OO_{Y^\circ}\to V\to T_{Y^\circ}\otimes L^{-1}\to0$;
        \item $Y$ is a point. Then $\rho(X)=1$, $-(K_X+\Delta)\equiv mN$, and $-(K_X+\Delta)$ is ample.
    \end{enumerate}
\end{enumerate}
\end{theorem}

The second result concerns the numerically trivial case. Under natural metric hypotheses on the open part $M:=X\setminus \Delta$, a logarithmic conformal structure with trivial conformal line bundle forces the restricted holonomy to be trivial. If, in addition, the induced flat connection has trivial monodromy, then $M$ is semi-abelian and $(X,\Delta)$ is its toroidal compactification.

\begin{theorem}\label{thm:semiabelian-intro}
Assume $K_X+\Delta\equiv 0$ and that $M$ carries a complete Ricci-flat K{\"a}hler metric $\omega$ satisfying \hyperlink{Bochner-extension}{\textup{(B)}} (Bochner extension) and \hyperlink{holonomy-condition}{\textup{(H)}} (restricted-holonomy condition). Assume moreover that $(X,\Delta)$ carries a log-conformal tensor $(N,g)$ with $N\simeq \OO_X$, and that the resulting flat connection on $T_M$ has trivial monodromy. Then $M$ is semi-abelian and $(X,\Delta)$ is a toroidal compactification of $M$.
\end{theorem}

In the third non-nef alternative, the integer $s$ counts the horizontal boundary components meeting a general minimal logarithmic curve at its unique boundary point; when the base is positive-dimensional, it determines both the fiber and the dimension of the base. The non-nef proof combines the determinant identity for the conformal line bundle with Langer's logarithmic $\A^1$-normal form, deformation estimates, and the self-duality of the logarithmic tangent bundle. These ingredients separate the three branches: the boundary-free case yields the quadric, conformal degree two with one boundary point yields $(\PP^n,H)$, and conformal degree one yields the extremal contraction with hyperplane-arrangement fibers. In the numerically trivial case, Bochner extension makes the conformal tensor parallel; the holonomy condition gives flatness, and trivial monodromy produces the logarithmic parallelisation used in the semi-abelian uniformization.

Section~\ref{sec:setup} introduces log-conformal structures, the numerical constraint on the conformal line bundle, the deformation theory of $(K_X+\Delta)$-minimal rational curves, and the logarithmic $\A^1$-normal form, together with the basic model geometries. Section~\ref{sec:splitting} records the conformal degree and the self-dual splitting constraints along a minimal rational curve. Section~\ref{sec:nonnef} proves the non-nef classification. Section~\ref{sec:trivial} treats the numerically trivial case under the stated Bochner, holonomy, and monodromy hypotheses. Section~\ref{subsec:cartan-interpretation} gives the logarithmic Cartan-geometric interpretation, and Section~\ref{subsec:boundary-contractions} treats boundary contractions, reflexive conformal structures on singular targets, and explicit surface examples.

\medskip\paragraph{\bf Acknowledgements.}
MC is grateful for the hospitality of the University of Ferrara during his visit in 2019, when this work was initiated; he is partially supported by the Università degli Studi di Bari and is a member of GNSAGA (INdAM). AM was supported by the PRIN 2022 grant (project 20223B5S8L), \emph{Birational geometry of moduli spaces and special varieties}; he is a member of GNSAGA (INdAM).

\section{Log-conformal structures, minimal curves and \texorpdfstring{$\A^1$}{A1}-normal form}\label{sec:setup}

\begin{definition}\label{def:logconf}
Let $(X,\Delta)$ be a smooth projective simple normal crossing pair, $\dim X=n\ge 3$.
Set $E:=T_X(-\log\Delta)$ and $E^\vee:=\Om_X^1(\log\Delta)$. A \emph{log-conformal structure} on $(X,\Delta)$ is a pair $(N,g)$ where $N$ is a line bundle on $X$ and $g\in H^0\bigl(X,\Sym^2(E^\vee)\otimes N\bigr)$ is fiberwise nondegenerate. That is, contraction with $g$ induces an isomorphism
\begin{equation}\label{eq:sharp}
g^\sharp:E \xrightarrow{\sim} E^\vee\otimes N.
\end{equation}
\end{definition}

\begin{lemma}\label{lem:det}
Let $(N,g)$ be a log-conformal structure as in Definition~\ref{def:logconf}. Then $N^{\otimes n}\simeq \det(E)^{\otimes 2}\simeq \OO_X\bigl(-2(K_X+\Delta)\bigr).$
For any morphism $f:\PP^1\to X$,
\begin{equation}\label{eq:degN}
n\cdot \deg(f^*N)=2\cdot\bigl(-(K_X+\Delta)\cdot f_*[\PP^1]\bigr).
\end{equation}
\end{lemma}

\begin{proof}
Taking determinants in \eqref{eq:sharp} gives $\det(E)\simeq \det(E^\vee)\otimes N^{\otimes n}\simeq \det(E)^{-1}\otimes N^{\otimes n},$
hence $N^{\otimes n}\simeq \det(E)^{\otimes 2}$.
For a simple normal crossing divisor $\Delta$, $\det\Om_X^1(\log\Delta)\simeq \OO_X(K_X+\Delta)$, hence $\det E\simeq \OO_X(-(K_X+\Delta))$.
Pulling back to $\PP^1$ and taking degrees gives \eqref{eq:degN}.
\end{proof}

We first describe the models appearing in the classification theorem.

\begin{proposition}\label{prop:quadric-example}
Set $q(z_0,\dots,z_{n+1}) := z_0^2+\cdots+z_{n+1}^2$ on $\C^{n+2}$, and let $X:=Q^n:=\{[z_0:\cdots:z_{n+1}]\in \PP^{n+1}\mid q=0\}.$
Assume $n\ge 3$ and that $Q^n$ is smooth. Set $\Delta:=\varnothing$ and $N:=\OO_{Q^n}(2)$.
Then $(Q^n,\varnothing)$ carries a natural log-conformal structure in the sense of Definition~\ref{def:logconf}, namely $g_Q\in H^0\bigl(Q^n,\Sym^2\Om_{Q^n}^1\otimes \OO_{Q^n}(2)\bigr)$, and $g_Q$ is fiberwise nondegenerate. Moreover $K_{Q^n}\simeq \OO_{Q^n}(-n),
\,
N^{\otimes n}\simeq \OO_{Q^n}(2n)\simeq \OO_{Q^n}\bigl(-2K_{Q^n}\bigr)$.
\end{proposition}

\begin{proof}
$K_{Q^n}\simeq \OO_{Q^n}(-n)$ and $\OO_{Q^n}\bigl(-2K_{Q^n}\bigr)\simeq \OO_{Q^n}(2n)=N^{\otimes n}$. Consider the affine cone $\widehat Q:=\{z=(z_0,\dots,z_{n+1})\in \C^{n+2}\setminus\{0\}\mid q(z)=0\},$
and let $p:\widehat Q\longrightarrow Q^n$ be the quotient by scalar multiplication. On $\C^{n+2}$ consider the constant symmetric bilinear form $\beta:=dz_0^2+\cdots+dz_{n+1}^2$. Fix $z\in \widehat Q$. Since $dq_z(v)=2\sum_{i=0}^{n+1} z_i v_i$, the tangent space to the cone at $z$ is
$$
T_z\widehat Q=\left\{v\in \C^{n+2}\ \middle|\ \sum_{i=0}^{n+1} z_i v_i=0\right\}.
$$
Since $q(z)=0$, the radial vector $z$ itself belongs to $T_z\widehat Q$.
The restriction $\beta|_{T_z\widehat Q}$ has radical exactly $\C z$. Indeed, if $w\in T_z\widehat Q$ satisfies $\beta(w,v)=0\, \text{for every } v\in T_z\widehat Q$, then $w$ is orthogonal to the hyperplane $T_z\widehat Q=z^\perp:=\left\{v\in \C^{n+2}\ \middle|\ \sum z_i v_i=0\right\}$
with respect to the standard symmetric form on $\C^{n+2}$. Hence $w\in (z^\perp)^\perp=\C z$. Conversely, since $z\in z^\perp$ and $q(z)=0$, $\beta(z,v)=\sum z_i v_i=0 \,\text{for every }v\in T_z\widehat Q$. Thus $\mathrm{rad}\bigl(\beta|_{T_z\widehat Q}\bigr)=\C z$. Hence $\beta$ induces a nondegenerate symmetric bilinear form on the quotient $T_z\widehat Q/\C z \simeq T_{p(z)}Q^n$. Under scalar multiplication $z\mapsto \lambda z$, the differentials transform as $dz_i\mapsto \lambda\, dz_i$.
Hence $\beta \mapsto \lambda^2 \beta$, and the bilinear form induced on the quotient $T_z\widehat Q/\C z$ is not invariant, but transforms with weight $2$. In other words, it descends to a global section $g_Q\in H^0\bigl(Q^n,\Sym^2\Om_{Q^n}^1\otimes \OO_{Q^n}(2)\bigr)$. Fiberwise nondegeneracy follows from the fact that $\beta|_{T_z\widehat Q}$ has radical exactly $\C z$.
\end{proof}

\begin{remark}\label{rem:quadric-local}
The tensor of Proposition~\ref{prop:quadric-example} has the following local expression. Work on the standard affine chart $U_0:=\{z_0\neq 0\}\subset Q^n, \qquad u_i:=\frac{z_i}{z_0}\quad (i=1,\dots,n+1).$
Then $U_0$ is cut out by $1+u_1^2+\cdots+u_{n+1}^2=0$. Fix the open subset $U_{0,n+1}:=U_0\cap \{u_{n+1}\neq 0\}$. On $U_{0,n+1}$, the implicit function theorem expresses $u_{n+1}$ as a holomorphic function of $(u_1,\dots,u_n)$, since $\frac{\partial}{\partial u_{n+1}}\bigl(1+u_1^2+\cdots+u_{n+1}^2\bigr)=2u_{n+1}\neq 0$. Write $u_{n+1}=h(u_1,\dots,u_n)$. Differentiating the equation $1+u_1^2+\cdots+u_n^2+h^2=0$ gives $2u_1\,du_1+\cdots+2u_n\,du_n+2h\,dh=0$, hence
\begin{equation}\label{eq:dh-on-quadric}
dh=-\frac{u_1}{h}\,du_1-\cdots-\frac{u_n}{h}\,du_n.
\end{equation}
Trivialize $\OO_{Q^n}(2)$ on $U_0$ by the section $z_0^2$. In this trivialization, the conformal tensor $g_Q$ is represented by the ordinary symmetric $2$-form $\widetilde g_Q:=du_1^2+\cdots+du_n^2+dh^2$. Using \eqref{eq:dh-on-quadric}, we obtain $dh^2=\left(\sum_{i=1}^n \frac{u_i}{h}\,du_i\right)^2
=\sum_{i,j=1}^n \frac{u_i u_j}{h^2}\,du_i\,du_j$. Therefore
$$
\widetilde g_Q
=
\sum_{i,j=1}^n
\left(
\delta_{ij}+\frac{u_i u_j}{h^2}
\right)
du_i\,du_j.
$$
The matrix of $\widetilde g_Q$ in the coordinate basis $(du_1,\dots,du_n)$ is $I_n+\frac{1}{h^2}\,u\,{}^t\!u,
\, u=(u_1,\dots,u_n)^t$. Its determinant is
$$
\det\!\left(I_n+\frac{1}{h^2}\,u\,{}^t\!u\right)
=
1+\frac{u_1^2+\cdots+u_n^2}{h^2}
=
\frac{h^2+u_1^2+\cdots+u_n^2}{h^2}
=
-\frac{1}{h^2}\neq 0,
$$
since $1+u_1^2+\cdots+u_n^2+h^2=0$. Thus the tensor is nondegenerate.
\end{remark}

\begin{remark}\label{rem:quadric-lines}
Let $\ell\subset Q^n$ be a line. Since $\Delta=\varnothing$, the logarithmic tangent bundle is the ordinary tangent bundle. The normal exact sequence $0\to N_{\ell/Q^n}\to N_{\ell/\PP^{n+1}}\to N_{Q^n/\PP^{n+1}}|_\ell\to 0$ becomes $0\to N_{\ell/Q^n}\to \OO_{\PP^1}(1)^{\oplus n}\to \OO_{\PP^1}(2)\to 0$, since $N_{\ell/\PP^{n+1}}\simeq \OO_{\PP^1}(1)^{\oplus n}$ and $N_{Q^n/\PP^{n+1}}\simeq \OO_{Q^n}(2)$. The standard normal-bundle calculation for a line on a smooth quadric gives $N_{\ell/Q^n}\simeq \OO_{\PP^1}(1)^{\oplus(n-2)}\oplus \OO_{\PP^1}$. Using $0\to T_\ell \to T_{Q^n}|_\ell \to N_{\ell/Q^n}\to 0$ and $T_\ell\simeq \OO_{\PP^1}(2)$, we obtain $T_{Q^n}|_\ell \simeq \OO_{\PP^1}(2)\oplus \OO_{\PP^1}(1)^{\oplus(n-2)}\oplus \OO_{\PP^1}$.
\end{remark}

We now consider projective space with a hyperplane boundary.

\begin{proposition}\label{prop:projective-example}
Let $X:=\PP^n$, let $\Delta:=H:=\{Z_0=0\}\subset \PP^n$, and set $N:=\OO_{\PP^n}(2)$.
Then $(\PP^n,H)$ carries a natural log-conformal structure $g_H\in H^0\bigl(\PP^n,\Sym^2\Om_{\PP^n}^1(\log H)\otimes \OO_{\PP^n}(2)\bigr)$, and $g_H$ is fiberwise nondegenerate. Moreover $K_{\PP^n}\simeq \OO_{\PP^n}(-n-1),\, K_{\PP^n}+H\simeq \OO_{\PP^n}(-n)$, so $N^{\otimes n}\simeq \OO_{\PP^n}(2n)\simeq \OO_{\PP^n}\bigl(-2(K_{\PP^n}+H)\bigr)$.
\end{proposition}

\begin{proof}
For $i=1,\dots,n$ define
\begin{equation}\label{eq:omegai-global}
\omega_i:=\frac{Z_0\,dZ_i-Z_i\,dZ_0}{Z_0}.
\end{equation}
Each $\omega_i$ is a global section of $\Om_{\PP^n}^1(\log H)\otimes \OO_{\PP^n}(1)$. Indeed, the numerator has homogeneous degree $2$ in the affine cone variables, and dividing by $Z_0$ lowers the weight by one; the only possible pole is along $Z_0=0$, and it is logarithmic. Set $z_i:=\frac{Z_i}{Z_0}$ for $i=1,\dots,n$. Then $U_0\simeq \A^n$ and $H\cap U_0=\varnothing$.
In the trivialization of $\OO_{\PP^n}(1)$ given by $Z_0$, one computes $\omega_i=d\!\left(\frac{Z_i}{Z_0}\right)=dz_i$. Thus on $U_0$ the forms $\omega_1,\dots,\omega_n$ are exactly the standard coordinate $1$-forms.

Work on $U_1:=\{Z_1\neq 0\}$. Use coordinates $u_0:=Z_0/Z_1,\, u_i:=Z_i/Z_1,\, i=2,\dots,n$. Then $H\cap U_1=\{u_0=0\}$. In the trivialization of $\OO_{\PP^n}(1)$ given by $Z_1$, the first form is $\omega_1=\frac{Z_0\,dZ_1-Z_1\,dZ_0}{Z_0}.$ Since $Z_0=u_0Z_1$, we have $dZ_0=du_0\cdot Z_1+u_0\,dZ_1$, hence $Z_0\,dZ_1-Z_1\,dZ_0
=
u_0 Z_1 dZ_1 - Z_1(du_0\cdot Z_1+u_0 dZ_1)
=
- Z_1^2\,du_0$. Dividing by $Z_0=u_0Z_1$ gives
\begin{equation}\label{eq:omega1-local}
\omega_1 = -\,\frac{du_0}{u_0}
\qquad
\text{in the local frame }Z_1.
\end{equation}
For $i\ge 2$, since $Z_i=u_i Z_1$, we have $dZ_i=du_i\cdot Z_1+u_i\,dZ_1$. Therefore
$$
Z_0\,dZ_i-Z_i\,dZ_0
=
u_0 Z_1(du_i\cdot Z_1+u_i dZ_1)-u_i Z_1(du_0\cdot Z_1+u_0 dZ_1)
=
Z_1^2(u_0\,du_i-u_i\,du_0),
$$
and dividing by $Z_0=u_0Z_1$ gives
\begin{equation}\label{eq:omegai-local}
\omega_i
=
du_i-\frac{u_i}{u_0}\,du_0
=
du_i-u_i\,\frac{du_0}{u_0},
\qquad i=2,\dots,n.
\end{equation}

Equations \eqref{eq:omega1-local} and \eqref{eq:omegai-local} show explicitly that the $\omega_i$ have only logarithmic poles along $H$. On $U_0$, this is clear since $\omega_i=dz_i$. On $U_1$, the local generators of $\Om_{\PP^n}^1(\log H)\otimes \OO_{\PP^n}(1)$ are $\frac{du_0}{u_0},\ du_2,\ \dots,\ du_n$. Equations~\eqref{eq:omega1-local} and \eqref{eq:omegai-local} recover these generators from the $\omega_i$: $\frac{du_0}{u_0}=-\omega_1,
\, du_i=\omega_i-u_i\omega_1,\, i=2,\dots,n$. Hence $\omega_1,\dots,\omega_n$ form a basis of $\Om_{\PP^n}^1(\log H)\otimes \OO_{\PP^n}(1)$ on $U_1$. Hence $\Om_{\PP^n}^1(\log H)\otimes \OO_{\PP^n}(1)\simeq \OO_{\PP^n}^{\oplus n}$, and after dualising,
\begin{equation}\label{eq:TPn-log-splitting}
T_{\PP^n}(-\log H)\simeq \OO_{\PP^n}(1)^{\oplus n}.
\end{equation}
Define
$$
g_H:=\omega_1^2+\cdots+\omega_n^2
\in H^0\bigl(\PP^n,\Sym^2\Om_{\PP^n}^1(\log H)\otimes \OO_{\PP^n}(2)\bigr).
$$
Since the $\omega_i$ form a local frame of $\Om_{\PP^n}^1(\log H)\otimes \OO_{\PP^n}(1)$, the matrix of $g_H$ in this frame is the identity matrix.
Hence $g_H$ is fiberwise nondegenerate.
\end{proof}

\begin{remark}\label{rem:projective-local}
On the affine chart $U_0=\{Z_0\neq 0\}$, where $\omega_i=dz_i$, the tensor of Proposition~\ref{prop:projective-example} is simply $g_H=dz_1^2+\cdots+dz_n^2$. Thus on $\PP^n\setminus H\simeq \A^n$ it is the standard flat quadratic form. On the boundary chart $U_1=\{Z_1\neq 0\}$, using \eqref{eq:omega1-local} and \eqref{eq:omegai-local}, one gets
\begin{equation}\label{eq:gH-local}
g_H
=
\left(\frac{du_0}{u_0}\right)^2
+
\sum_{i=2}^n
\left(du_i-u_i\frac{du_0}{u_0}\right)^2.
\end{equation}

\end{remark}

\begin{remark}\label{rem:projective-tangent-frame}
On $U_1$, define vector fields
$$
\xi_1:=-u_0\frac{\partial}{\partial u_0}-\sum_{i=2}^n u_i\frac{\partial}{\partial u_i},
\qquad
\xi_i:=\frac{\partial}{\partial u_i}\quad (i=2,\dots,n).
$$
These are logarithmic vector fields along $H\cap U_1=\{u_0=0\}$.
A direct computation gives $\omega_1(\xi_1)=1,\, \omega_i(\xi_1)=0,\, i\ge 2$, and $\omega_1(\xi_j)=0,\, \omega_i(\xi_j)=\delta_{ij},\, i,j\ge 2$. Thus $\xi_1,\dots,\xi_n$ form the local frame of $T_{\PP^n}(-\log H)$ dual to the coframe $\omega_1,\dots,\omega_n$. In this frame, the tensor $g_H$ is represented by the identity matrix.
\end{remark}

\begin{remark}\label{rem:projective-lines}
Let $\ell\subset \PP^n$ be a line not contained in $H$, and $f:\PP^1\to \ell\subset \PP^n$ a parametrisation. Then $\ell$ meets $H$ in exactly one point and the intersection is reduced. Hence $(f^{-1}H)_{\mathrm{red}}$ is a single point and $\delta=\deg\bigl((f^{-1}H)_{\mathrm{red}}\bigr)=1$. Moreover, by \eqref{eq:TPn-log-splitting}, $f^*T_{\PP^n}(-\log H)\simeq \OO_{\PP^1}(1)^{\oplus n}$.
\end{remark}

\begin{proposition}\label{prop:abelian-projective-example}
Let $A$ be an abelian variety of dimension $r\ge 1$, let $p_1:A\times \PP^r\to A$ and $p_2:A\times \PP^r\to \PP^r$ be the projections, and let $H\subset \PP^r$ be a hyperplane. Set $X:=A\times \PP^r$, $\Delta:=A\times H$, and
$N:=p_2^*\OO_{\PP^r}(1)$.
Then $(X,\Delta)$ carries an everywhere nondegenerate log-conformal tensor with values in $N$, and $K_X+\Delta = -r\,p_2^*H$. In particular $K_X+\Delta$ is not nef.
\end{proposition}

\begin{proof}
Since $T_A\simeq \OO_A^{\oplus r}$ and Proposition~\ref{prop:projective-example} gives $T_{\PP^r}(-\log H)\simeq \OO_{\PP^r}(1)^{\oplus r}$, it follows that $$
T_X(-\log\Delta)
\simeq p_1^*T_A\oplus p_2^*T_{\PP^r}(-\log H)
\simeq \OO_X^{\oplus r}\oplus N^{\oplus r}.
$$
Let $(e_1,\dots,e_r)$ be the standard frame of the trivial summand and let $(\eta_1,\dots,\eta_r)$ be the tautological frame of $N^{\oplus r}$. Define a symmetric $N$-valued bilinear form on $T_X(-\log\Delta)$ by $g\bigl((u,\varphi),(u',\varphi')\bigr):=\varphi'(u)+\varphi(u')$. In the frame $(e_1,\dots,e_r,\eta_1,\dots,\eta_r)$ its matrix is
$$
\begin{pmatrix}
0 & I_r\\
I_r & 0
\end{pmatrix},
$$
hence $g$ is fiberwise nondegenerate. Thus $(X,\Delta)$ carries a log-conformal structure. Since $K_A\simeq \OO_A$ and $K_{\PP^r}+H\simeq \OO_{\PP^r}(-r)$, $K_X+\Delta = p_1^*K_A+p_2^*(K_{\PP^r}+H) = -r\,p_2^*H$, which is not nef.
\end{proof}

\begin{remark}\label{rem:fibrational-model}
Proposition~\ref{prop:abelian-projective-example} exhibits the $s=1$ geometric pattern underlying the split-null alternative of the main theorem: the pair $(A\times \PP^r, A\times H)$ carries a nondegenerate log-conformal tensor, and the projection to $A$ has general fiber $\PP^r$ with hyperplane boundary.
\end{remark}

\begin{proposition}\label{prop:multi-hyperplane-example}
Let $m\ge1$ and $1\le s\le m+1$. Let $A$ be an abelian variety of dimension $m-s+1$ (a point if $s=m+1$), let $H_1,\dots,H_s\subset \PP^{m+s-1}$ be hyperplanes in general position, and set
$$
X:=A\times\PP^{m+s-1},\qquad
\Delta:=A\times(H_1+\cdots+H_s),\qquad
N:=p_2^*\OO_{\PP^{m+s-1}}(1).
$$
Then $\dim X=2m$, the pair $(X,\Delta)$ carries an everywhere nondegenerate log-conformal tensor with values in $N$, and $K_X+\Delta=-m\,p_2^*H.$
For $s=m+1$ this gives the Picard-rank-one model $(\PP^{2m},H_1+\cdots+H_{m+1}).$
\end{proposition}

\begin{proof}
For an arrangement of $s$ hyperplanes in general position, \cite[Remark~1.2]{Langer24} gives
$$
T_{\PP^{m+s-1}}\bigl(-\log(H_1+\cdots+H_s)\bigr)
\simeq \OO_{\PP^{m+s-1}}^{\oplus(s-1)}\oplus
\OO_{\PP^{m+s-1}}(1)^{\oplus m}.
$$
Since $T_A\simeq\OO_A^{\oplus(m-s+1)}$, it follows that $T_X(-\log\Delta)\simeq \OO_X^{\oplus m}\oplus N^{\oplus m}.$
The hyperbolic pairing between the two summands, as in Proposition~\ref{prop:abelian-projective-example}, is a fiberwise nondegenerate symmetric $N$-valued form. Finally, $K_{\PP^{m+s-1}}+H_1+\cdots+H_s=-(m+s)H+sH=-mH,$
and $K_A\simeq\OO_A$, proving the formula for $K_X+\Delta$.
\end{proof}

\begin{remark}\label{rem:projective-general-matrix}
More generally, fix any invertible symmetric matrix $A=(a_{ij})\in \mathrm{Mat}_{n\times n}(\C), \qquad A={}^{t}\!A, \qquad \det(A)\neq 0,$
and define
$$
g_{H,A}:=\sum_{i,j=1}^n a_{ij}\,\omega_i\,\omega_j
\in H^0\bigl(\PP^n,\Sym^2\Om_{\PP^n}^1(\log H)\otimes \OO_{\PP^n}(2)\bigr).
$$
In the frame $\omega_1,\dots,\omega_n$, the matrix of $g_{H,A}$ is exactly $A$.
Thus $g_{H,A}$ is fiberwise nondegenerate if and only if $A$ is invertible.
The model of Proposition~\ref{prop:projective-example} is the special case $A=I_n$.
\end{remark}

For the two rank-one models, the isomorphism $g^\sharp:T_X(-\log\Delta)\xrightarrow{\sim}\Om_X^1(\log\Delta)\otimes N$
from Definition~\ref{def:logconf} describes nondegeneracy on the tangent side.

\begin{proposition}\label{prop:projective-sharp}
Let $(X,\Delta)=(\PP^n,H)$ be as in Proposition~\ref{prop:projective-example}, with $H=\{Z_0=0\},\qquad N=\OO_{\PP^n}(2), \qquad g_H=\omega_1^2+\cdots+\omega_n^2.$
Then contraction with $g_H$ induces an isomorphism $g_H^\sharp:T_{\PP^n}(-\log H)\xrightarrow{\sim}\Om_{\PP^n}^1(\log H)\otimes \OO_{\PP^n}(2)$. Moreover this isomorphism is diagonal in the natural logarithmic frames.
\end{proposition}

\begin{proof}
By Remark~\ref{rem:projective-tangent-frame}, on the chart $U_1=\{Z_1\neq 0\}$ the logarithmic tangent bundle $T_{\PP^n}(-\log H)$ admits the local frame
$$
\xi_1:=-u_0\frac{\partial}{\partial u_0}-\sum_{i=2}^n u_i\frac{\partial}{\partial u_i},
\qquad
\xi_i:=\frac{\partial}{\partial u_i},\quad i=2,\dots,n,
$$
which is dual to the logarithmic coframe $\omega_1=-\frac{du_0}{u_0}, \qquad \omega_i=du_i-u_i\frac{du_0}{u_0},\quad i=2,\dots,n.$
Since $g_H=\omega_1^2+\cdots+\omega_n^2$, any logarithmic vector field $\eta$ satisfies
$$
g_H^\sharp(\eta)=g_H(\eta,-)=\sum_{i=1}^n\omega_i(\eta)\,\omega_i.
$$
In particular, $g_H^\sharp(\xi_j)=\omega_j$ for $j=1,\dots,n$. Thus the matrix of $g_H^\sharp$ in the frames $(\xi_1,\dots,\xi_n)$ and $(\omega_1,\dots,\omega_n)$ is the identity matrix, and $g_H^\sharp$ is an isomorphism on $U_1$.
The same holds on the affine chart $U_0=\{Z_0\neq 0\}$, where $(dz_1,\dots,dz_n)$ is dual to $(\partial/\partial z_1,\dots,\partial/\partial z_n)$ and $g_H^\sharp(\partial/\partial z_i)=dz_i$. Hence the map is again represented by the identity matrix.
Since the local matrices are invertible everywhere, $g_H^\sharp$ is globally an isomorphism. This proves that $g_H$ is fiberwise nondegenerate in the sense of Definition~\ref{def:logconf}.
\end{proof}

\begin{remark}\label{rem:projective-sharp-matrix}
On $U_1=\{Z_1\neq 0\}$, in the logarithmic tangent frame $(\xi_1,\xi_2,\dots,\xi_n)$ and the logarithmic cotangent frame $(\omega_1,\omega_2,\dots,\omega_n)$, the isomorphism $g_H^\sharp$ is represented by $[g_H^\sharp]=I_n$. More generally, for the tensor $g_{H,A}$ of Remark~\ref{rem:projective-general-matrix},
the induced map
$$
g_{H,A}^\sharp:T_{\PP^n}(-\log H)\longrightarrow \Om_{\PP^n}^1(\log H)\otimes \OO_{\PP^n}(2)
$$
is represented in these frames by the symmetric matrix $A$.
Hence $g_{H,A}^\sharp$ is an isomorphism if and only if $\det(A)\neq 0$.
\end{remark}

\begin{proposition}\label{prop:quadric-sharp}
Let $(X,\Delta)=(Q^n,\varnothing)$ be as in Proposition~\ref{prop:quadric-example}, with $N=\OO_{Q^n}(2)$, and let $g_Q\in H^0\bigl(Q^n,\Sym^2\Om_{Q^n}^1\otimes \OO_{Q^n}(2)\bigr)$
be the conformal tensor induced by the constant quadratic form on $\C^{n+2}$.
Then contraction with $g_Q$ induces an isomorphism $g_Q^\sharp:T_{Q^n}\xrightarrow{\sim}\Om_{Q^n}^1\otimes \OO_{Q^n}(2)$.
Locally, in coordinates $(u_1,\dots,u_n)$ on the quadric, the matrix of $g_Q^\sharp$ is the symmetric matrix $\left(\delta_{ij}+\frac{u_i u_j}{h^2}\right)_{1\le i,j\le n}$, where $u_{n+1}=h(u_1,\dots,u_n)$ is the local defining function from Remark~\ref{rem:quadric-local}.
\end{proposition}

\begin{proof}
By Remark~\ref{rem:quadric-local}, on the chart $U_{0,n+1}\subset Q^n$ with coordinates $(u_1,\dots,u_n)$, the tensor $g_Q$ is represented, after trivialising $\OO_{Q^n}(2)$ by $z_0^2$, by
$$
\widetilde g_Q
=
\sum_{i,j=1}^n
\left(
\delta_{ij}+\frac{u_i u_j}{h^2}
\right)
du_i\,du_j.
$$
Hence, for a tangent vector field $\eta=\sum_{j=1}^n a_j\frac{\partial}{\partial u_j}$,
$$
g_Q^\sharp(\eta)
=
\sum_{i,j=1}^n
\left(
\delta_{ij}+\frac{u_i u_j}{h^2}
\right)
a_j\,du_i.
$$
Thus, in the coordinate frames $\left(\frac{\partial}{\partial u_1},\dots,\frac{\partial}{\partial u_n}\right)$ and $(du_1,\dots,du_n)$, the matrix of $g_Q^\sharp$ is $G(u):=\left(\delta_{ij}+\frac{u_i u_j}{h^2}\right)_{1\le i,j\le n}.$
Remark~\ref{rem:quadric-local} gives $\det G(u) = 1+\frac{u_1^2+\cdots+u_n^2}{h^2} = -\frac{1}{h^2}\neq 0.$
Hence $g_Q^\sharp$ is an isomorphism on this chart, hence everywhere.
\end{proof}

\begin{remark}\label{rem:quadric-inverse}
The inverse matrix of $G(u)=I_n+\frac{1}{h^2}u\,{}^t\!u$ is given by the rank-one perturbation formula $G(u)^{-1} = I_n-\frac{u\,{}^t\!u}{h^2+u_1^2+\cdots+u_n^2}.$
Since the quadric equation gives $h^2+u_1^2+\cdots+u_n^2=-1$, this becomes $G(u)^{-1}=I_n+u\,{}^t\!u$. Thus
$$
\sum_{k=1}^n
\left(\delta_{ik}+\frac{u_i u_k}{h^2}\right)
\left(\delta_{kj}+u_k u_j\right)
=
\delta_{ij}.
$$
That is, the inverse bundle isomorphism is
$$
(g_Q^\sharp)^{-1}:\Om_{Q^n}^1\otimes \OO_{Q^n}(2)\longrightarrow T_{Q^n}.
$$
\end{remark}

\begin{remark}\label{rem:null-quadric-examples}
In both basic examples the pointwise null cone can be written explicitly. On the affine chart $U_0=\PP^n\setminus H\simeq \A^n$, where $g_H=dz_1^2+\cdots+dz_n^2$, the null directions at a point are exactly $[a_1:\cdots:a_n]\in \PP(T_x\A^n)\simeq \PP^{n-1}
\,\text{such that}\, a_1^2+\cdots+a_n^2=0$. Thus the null quadric is the standard smooth quadric in $\PP^{n-1}$.
On $Q^n$ consider the chart of Remark~\ref{rem:quadric-local}. A tangent vector $v=\sum_{i=1}^n a_i\frac{\partial}{\partial u_i}$ is null if and only if $\widetilde g_Q(v,v)
=
\sum_{i,j=1}^n
\left(
\delta_{ij}+\frac{u_i u_j}{h^2}
\right)
a_i a_j
=
\sum_{i=1}^n a_i^2+\frac{1}{h^2}\left(\sum_{i=1}^n u_i a_i\right)^2
=0$. Thus again the null directions form a smooth quadric hypersurface in the projectivized tangent space. This is the local incarnation of the VMRT on the quadric.
\end{remark}

These examples realize the geometric patterns that occur in the main theorem. The two rank-one models give the rigid cases $(Q^n,\varnothing)$ and $(\PP^n,H)$, while Proposition~\ref{prop:abelian-projective-example} illustrates the fibrational geometry of the remaining alternative.

Indeed, if $(X,\Delta)=(Q^n,\varnothing)$ and $\ell\subset Q^n$ is a line, then Remark~\ref{rem:quadric-lines} shows that $T_{Q^n}|_\ell \simeq \OO_{\PP^1}(2)\oplus \OO_{\PP^1}(1)^{\oplus (n-2)}\oplus \OO_{\PP^1}.$
Since the boundary is empty, this coincides with
$$
T_{Q^n}(-\log\Delta)|_\ell \simeq \OO_{\PP^1}(2)\oplus \OO_{\PP^1}(1)^{\oplus (n-2)}\oplus \OO_{\PP^1}.
$$
This is the tangent splitting along lines in the quadric model.

If $(X,\Delta)=(\PP^n,H)$ and $\ell\subset \PP^n$ is a line not contained in $H$, then Remark~\ref{rem:projective-lines} gives $$
T_{\PP^n}(-\log H)|_\ell \simeq \OO_{\PP^1}(1)^{\oplus n},\, \delta=\deg\bigl((f^{-1}H)_{\mathrm{red}}\bigr)=1.
$$ Thus the projective-space example has the logarithmic splitting used in the conformal-degree-two branch. Proposition~\ref{prop:abelian-projective-example} gives the $s=1$ member of the split-null family, while Proposition~\ref{prop:multi-hyperplane-example} exhibits the corresponding models with several boundary hyperplanes.


Assume henceforth that $K_X+\Delta$ is not nef.
By the cone and contraction theorem, there exists a $(K_X+\Delta)$-negative extremal ray
$R\subset \NE(X)$ and a contraction $\phi_R:X\to Y$ contracting precisely the curves with class in $R$ \cite[Chapter 3]{KM98}. The simple normal crossing formulation is \cite[Theorem~1.4]{Langer24}.

\begin{definition}\label{def:length}
The length of $R$ is
$$
\ell(R):=\min\bigl\{-(K_X+\Delta)\cdot C : C\text{ is a rational curve with }[C]\in R\bigr\}.
$$
\end{definition}

\begin{definition}\label{def:unsplit}
Let $X$ be a projective variety and let $V\subset\Hom(\PP^1,X)$ be an irreducible family. We say that $V$ is \emph{unsplit} if the curves parametrized by $V$ do not degenerate to a $1$-cycle consisting of more than one curve, counting multiplicities, \cite[Chapter~IV, Definition~2.1]{Kollar96}, \cite[Section~1]{Langer24}.
\end{definition}

Let $\Delta=\sum_{i\in I}\Delta_i$. For effective divisors $G_i$ on $\PP^1$, we denote by
$\Hom(\PP^1,X;G\subset\Delta)$ the closed subscheme parametrizing morphisms $h:\PP^1\to X$ such that $G_i\le h^*\Delta_i$ for every $i$, \cite[Section~2.1]{Langer24}. If $G_i=f^*\Delta_i$ for all $i$, its Zariski tangent space at $[f]$ is $H^0(\PP^1,f^*T_X(-\log\Delta))$.
For a family $V$ we write $\locus(V)$ for the union of its curves and $\locus(V,x)$ for the union of members through $x$; if $U\subset X$ is open, set $\locus_U(V):=\locus(V)\cap U$ and $\locus_U(V,x):=\locus(V,x)\cap U$. After choosing $0\in\PP^1$ away from the prescribed boundary support, write $V_x:=\{[g]\in V:g(0)=x\}$. Finally, if $E\simeq\bigoplus_i\OO_{\PP^1}(a_i)$, set
$\rk_+(E):=\#\{i:a_i>0\}$.

\begin{lemma}\label{lem:unsplit-minimal}
Fix a rational curve $C$ with $[C]\in R$ and $-(K_X+\Delta)\cdot C=\ell(R)$, let
$f:\PP^1\to X$ be the normalization of $C$, and let $V$ be an irreducible component of $\Hom(\PP^1,X)$ containing $[f]$.
Then $V$ is unsplit.
\end{lemma}

\begin{proof}
Assume, to the contrary, that $V$ is not unsplit.
By Definition~\ref{def:unsplit}, there exists a $1$-parameter degeneration of curves in $V$ whose limit cycle is reducible. In particular, after passing to the limit in the Chow variety, the cycle $C$ degenerates as $C\equiv C_1+\cdots+C_m$ with $m\ge 2$ and each $C_i$ an effective rational curve.

Since $[C]\in R$ and $R$ is an extremal ray, $[C_i]\in R$ for every $i=1,\dots,m$. Since $\Delta$ is reduced simple normal crossing, the divisor $K_X+\Delta$ is Cartier, hence $(K_X+\Delta)\cdot C_i\in \Z$. Moreover, since $(K_X+\Delta)\cdot R<0$, each number $-(K_X+\Delta)\cdot C_i$ is a positive integer. Hence $\ell(R)=-(K_X+\Delta)\cdot C
      =\sum_{i=1}^m\bigl(-(K_X+\Delta)\cdot C_i\bigr)$ expresses $\ell(R)$ as a sum of at least two positive integers.
Hence one of the summands is strictly smaller than $\ell(R)$, contradicting the definition of $\ell(R)$.
\end{proof}

\begin{proposition}\label{prop:A1-normal}
Assume that $K_X+\Delta$ is not nef, and let $R\subset \NE(X)$ be a $(K_X+\Delta)$-negative extremal ray. Let $C$ be a $(K_X+\Delta)$-minimal rational curve spanning $R$, and let $x\in C\cap (X\setminus \Delta)$.
Then there exists a morphism $g:\PP^1\to X$ such that $x\in g(\PP^1)$, $g_*[\PP^1]\in R$, and the reduced divisor $D_{\mathrm{red}}:=\bigl(g^{-1}\Delta\bigr)_{\mathrm{red}}$ is supported on at most one point. Moreover, $g$ can be chosen in the connected component of $\Hom(\PP^1,X)$ containing the normalization of $C$; in particular $-(K_X+\Delta)\cdot g_*[\PP^1]=\ell(R).$
Thus $g_*[\PP^1]$ is again minimal in $R$, and $\delta:=\deg D_{\mathrm{red}}\le1$.
\end{proposition}
\begin{proof}
Apply \cite[Proposition~2.1]{Langer24} with $D=\Delta$ and $J=\varnothing$. Then $D_J=X$ and $U_J=X\setminus\Supp\Delta$, so the proposition gives a morphism $g:\PP^1\to X$ through $x$, with $g_*[\PP^1]\in R$, such that $g^*\Delta$ is supported on at most one point.

In \cite[Proposition~2.1]{Langer24}, the morphism $g$ is constructed in the connected component $V\subset\Hom(\PP^1,X)$ containing the normalization $f$ of $C$. The degree of the pull-back of a line bundle is constant on $V$. Hence $\deg g^*A=\deg f^*A$ for every line bundle $A$ on $X$, so $g_*[\PP^1]\equiv C$ in $N_1(X)$ and $-(K_X+\Delta)\cdot g_*[\PP^1]=\ell(R).$
If $g$ had degree greater than one onto its image, the image curve would lie in $R$ and have strictly smaller $(K_X+\Delta)$-degree, contrary to the definition of $\ell(R)$. Hence $g$ is birational onto its image. Taking the reduced inverse image of $\Delta$ proves the assertion.
\end{proof}

\begin{proposition}\label{prop:Langer-estimates}
Let $f$ be the minimal morphism furnished by Proposition~\ref{prop:A1-normal}. Set $U:=X\setminus\Supp\Delta$ and $\delta:=\deg((f^{-1}\Delta)_{\rm red})\le1$.
\begin{enumerate}
\item If $\delta=1$ and $W$ is the irreducible component of $\Hom(\PP^1,X;f^*\Delta\subset\Delta)$ containing $[f]$, then $\dim W\ge n+\ell(R).$
If $W$ is unsplit, then for $x\in\locus_U(W)$,
$$
\dim W\le\dim\locus_U(W)+\dim\locus_U(W,x),\qquad
\dim\locus_U(W,x)\le\max_{[g]\in W_x}\rk_+\bigl(g^*T_X(-\log\Delta)\bigr).
$$
\item If $\delta=0$ and $V$ is the irreducible component of $\Hom(\PP^1,X)$ containing $[f]$, then $\dim V\ge n+\ell(R).$
If $V$ is unsplit, then for every $x\in\locus(V)\cap U$,
$$
\dim V\le\dim\locus(V)+\dim\locus(V,x)+1,
\qquad
\dim\locus(V,x)\le\max_{[g]\in V_x}\rk_+(g^*T_X).
$$
\end{enumerate}
\end{proposition}

\begin{proof}
For \textup{(i)}, the dimension estimate in \cite[Section~2.1]{Langer24} gives $\dim_{[f]}W\ge \chi(\PP^1,f^*T_X)-\deg(f^*\Delta) =n-(K_X+\Delta)\cdot f_*[\PP^1]=n+\ell(R).$
The subgroup of $\Aut(\PP^1)$ fixing the unique point of $\Supp(f^*\Delta)$ preserves $W$. When $W$ is unsplit, Propositions~2.4 and~2.5 of \cite{Langer24} give the two asserted upper bounds. For \textup{(ii)}, $f^*\Delta=0$, so $\dim_{[f]}V\ge\chi(\PP^1,f^*T_X)=n+\ell(R)$. Propositions~2.3 and~2.5 of \cite{Langer24} give the asserted upper bounds; for $x\in U$, every member of $V_x$ lies in $U$ since $\Delta\cdot R=0$.
\end{proof}


\section{Conformal degree along minimal rational curves}\label{sec:splitting}
Let $f:\PP^1\to X$ be the minimal morphism furnished by Proposition~\ref{prop:A1-normal}, and set
$E_f:=f^*T_X(-\log\Delta)$ and $d:=\deg(f^*N)$.

\begin{lemma}\label{lem:conformal-degree}
With the notation above, $d\in\{1,2\}$ and $n d=2\ell(R).$
If $E_f\simeq\bigoplus_{i=1}^n\OO_{\PP^1}(b_i)$ with $b_1\ge\cdots\ge b_n$, then $b_i+b_{n+1-i}=d\qquad(1\le i\le n).$
If $d=1$, then $n=2m$ is even and $\rk_+(E_f)=m$.
\end{lemma}

\begin{proof}
The first identity is Lemma~\ref{lem:det}. Applied with $J=\varnothing$, \cite[Proposition~1.5]{Langer24} says that $R$ contains a rational curve $\Gamma$ with
$0<-(K_X+\Delta)\cdot\Gamma\le n+1$. Hence $\ell(R)\le n+1$, and therefore
$0<d=2\ell(R)/n\le2+2/n<3$. Since $d$ is an integer, $d\in\{1,2\}$. Pulling back the conformal isomorphism $T_X(-\log\Delta)\simeq\Omega_X^1(\log\Delta)\otimes N$ gives $E_f\simeq E_f^\vee\otimes\OO_{\PP^1}(d)$. After ordering the splitting degrees, uniqueness of the Birkhoff--Grothendieck decomposition gives $b_i+b_{n+1-i}=d$. If $d=1$, an unpaired middle summand would have degree $1/2$; hence $n=2m$. In every pair $(b_i,b_{2m+1-i})$ the two integral degrees sum to one, so exactly one is positive. Hence $\rk_+(E_f)=m$.
\end{proof}

\section{Classification when \texorpdfstring{$K_X+\Delta$}{KX+Delta} is not nef}\label{sec:nonnef}
Let $R\subset\NE(X)$ be the fixed $(K_X+\Delta)$-negative extremal ray, and let $f:\PP^1\to X$ be the minimal morphism supplied by Proposition~\ref{prop:A1-normal}. Set $\delta:=\deg((f^{-1}\Delta)_{\rm red})\le1$.

\subsection*{The case $\delta=1$}
Write $(f^{-1}\Delta)_{\rm red}=p$. Let $\Delta_1,\dots,\Delta_s$ be the components of $\Delta$ through $f(p)$. In simple normal crossing coordinates $\Delta=\{x_1\cdots x_s=0\}$, write $f^*\Delta_i=a_i p$, $a_i\ge1$. Thus $f^*\Delta=(a_1+\cdots+a_s)p,\qquad \Delta\cdot f_*[\PP^1]=a_1+\cdots+a_s.$
In particular, $\delta=1$ does not imply $s=1$ or $\Delta\cdot f_*[\PP^1]=1$.

\begin{lemma}\label{lem:pullback-log-seq}
There is a natural exact sequence
$$
0\to f^*T_X(-\log\Delta)\to f^*T_X\to\bigoplus_{i=1}^s\OO_{a_i p}\to0.
$$
\end{lemma}

\begin{proof}
Locally $T_X(-\log\Delta)$ is generated by $x_i\partial_{x_i}$ for $1\le i\le s$ and by $\partial_{x_j}$ for $j>s$. If $t$ is a coordinate at $p$, then, after multiplying the $x_i$ by units, write $f^*x_i=t^{a_i}$. The quotient is therefore $\bigoplus_i\OO_{\PP^1,p}/(t^{a_i})$, and away from $p$ the logarithmic and ordinary tangent bundles coincide.
\end{proof}

\begin{lemma}\label{lem:delta1-covering}
Let $d:=\deg(f^*N)$ and let $W\subset\Hom(\PP^1,X;f^*\Delta\subset\Delta)$ be the component containing $[f]$. Then $W$ is unsplit and covering. More precisely:
\begin{enumerate}
\item if $d=2$, then $\dim W=2n$, and for a general member $g$ of $W$,
$g^*T_X(-\log\Delta)\simeq\OO_{\PP^1}(1)^{\oplus n}$;
\item if $d=1$, then $n=2m$, $\dim W=3m$, and for general $x\in U:=X\setminus\Supp\Delta$, $\dim\locus_U(W)=2m,\qquad \dim\locus_U(W,x)=m.$
\end{enumerate}
\end{lemma}

\begin{proof}
Let $V$ be an irreducible component of $\Hom(\PP^1,X)$ containing $W$. Minimality gives that $V$ is unsplit by Lemma~\ref{lem:unsplit-minimal}, hence so is $W$. Proposition~\ref{prop:Langer-estimates}\textup{(i)} gives $\dim W\ge n+\ell(R)=n+\frac n2d.$
If $d=2$, this lower bound is $2n$, while Langer's upper bounds give $\dim W\le\dim\locus_U(W)+\max\rk_+\le n+n=2n.$
Thus equality holds. In particular $\dim\locus_U(W)=n$, so $W$ is covering, and for general $x$ one also has $\dim\locus_U(W,x)=n$. By generic smoothness and \cite[Proposition~2.5]{Langer24}, the pointed evaluation has differential of rank $n$ at a general point of $W_x$. Maximal rank is an open condition, hence $\rk_+(g^*T_X(-\log\Delta))=n$ for a general member $g$. Moreover $\deg g^*T_X(-\log\Delta)=-(K_X+\Delta)\cdot g_*[\PP^1]=\ell(R)=n.$
All $n$ summands are positive and their degrees add to $n$, so each has degree one.

If $d=1$, Lemma~\ref{lem:conformal-degree} gives $n=2m$ and $\rk_+(g^*T_X(-\log\Delta))=m$ for every $g$ in the component. Hence $3m\le\dim W\le\dim\locus_U(W)+\dim\locus_U(W,x)\le2m+m,$
and equality holds throughout. Hence $W$ is covering and $\dim W=3m$. Moreover $\dim\locus_U(W)=2m$ and $\dim\locus_U(W,x)=m$.
\end{proof}

\begin{lemma}\label{lem:delta1-anticanonical}
Assume $\delta=1$ and $\deg(f^*N)=2$. Then $R$ is $K_X$-negative, $f_*[\PP^1]$ is $K_X$-minimal in $R$, and $-K_X\cdot f_*[\PP^1]=n+1,\qquad \Delta\cdot f_*[\PP^1]=1.$
\end{lemma}

\begin{proof}
Here $\ell(R)=n$. Set $q:=\Delta\cdot f_*[\PP^1]\ge1$. Then $-K_X\cdot f_*[\PP^1]=n+q$. If $\Gamma$ is another rational curve in $R$, write $[\Gamma]=a\,f_*[\PP^1]$ in $N_1(X)$. Since $-(K_X+\Delta)\cdot\Gamma=an$ is a positive integer and $n=\ell(R)$, $a\ge1$; hence $f_*[\PP^1]$ is also $K_X$-minimal. Mori's length bound gives $n+q\le n+1$, so $q=1$.
\end{proof}

\begin{corollary}\label{cor:delta1-TX-split}
Under the hypotheses of Lemma~\ref{lem:delta1-anticanonical}, for a general member $f$ of $W$, $f^*T_X\simeq\OO_{\PP^1}(2)\oplus\OO_{\PP^1}(1)^{\oplus(n-1)}.$
\end{corollary}

\begin{proof}
By Lemma~\ref{lem:delta1-covering}\textup{(i)}, $f^*T_X(-\log\Delta)\simeq\OO(1)^{\oplus n}$. Lemma~\ref{lem:delta1-anticanonical} and Lemma~\ref{lem:pullback-log-seq} give
$0\to\OO(1)^{\oplus n}\to f^*T_X\to k(p)\to0$, so $\deg f^*T_X=n+1$. If $f^*T_X=\bigoplus_i\OO(d_i)$ and some $d_i\le0$, the projection onto that summand annihilates $\OO(1)^{\oplus n}$, forcing the latter to inject generically into a bundle of rank $n-1$, a contradiction. Thus all $d_i\ge1$, while $df\ne0$ gives some $d_i\ge2$; the degree forces the displayed splitting.
\end{proof}

\begin{theorem}\label{thm:Pn}
If $\delta=1$ and $\deg(f^*N)=2$, then $X\simeq\PP^n$ and $\Delta$ is a hyperplane.
\end{theorem}

\begin{proof}
Lemma~\ref{lem:delta1-covering} gives a dominating unsplit family, hence a proper family of rational curves in the Chow variety. For a general member through a general point $x$, Corollary~\ref{cor:delta1-TX-split} gives $f^*T_X\simeq\OO(2)\oplus\OO(1)^{\oplus(n-1)}.$
Consequently $H^1(\PP^1,f^*T_X(-1))=0$ and $h^0(\PP^1,f^*T_X(-1))=n+1$. Quotienting the pointed morphisms by $\Aut(\PP^1,0)$ leaves an $(n-1)$-dimensional family of curves through $x$. The characterisations of projective space in \cite[Theorem~1.1]{Keb02} and \cite[Theorem~0.1]{CMSB02} therefore give $X\simeq\PP^n$. By Lemma~\ref{lem:delta1-anticanonical}, a general member has anticanonical degree $n+1$ and is therefore a line in $\PP^n$; since $\Delta\cdot C=1$, the divisor $\Delta$ has degree one and is a hyperplane.
\end{proof}

\begin{theorem}\label{thm:third-branch}
Assume $\delta=1$ and $\deg(f^*N)=1$. Let $\phi_R:X\to Y$ be the contraction of $R$ and set $b:=\dim Y$. Then $n=2m$, $\phi_R$ is of fiber type, and one of the following occurs.
\begin{enumerate}
\item $b>0$. For some $1\le s\le m$, $b=m-s+1,$
and for a general fiber $F$, $(F,\Delta|_F,N|_F)\simeq (\PP^{m+s-1},H_1+\cdots+H_s,\OO_{\PP^{m+s-1}}(1)),$
where the $H_i$ are hyperplanes in general position. A general minimal curve $C\subset F$ is a line and $s=\Delta\cdot C$. Along such a line,
$$
T_X(-\log\Delta)|_C\simeq\OO(1)^{\oplus m}\oplus\OO^{\oplus m},\qquad
T_X|_C\simeq\OO(2)\oplus\OO(1)^{\oplus(m+s-2)}\oplus\OO^{\oplus(m-s+1)}.
$$
After shrinking $Y$, $\phi_R$ is a projective bundle and the horizontal part of $\Delta$ is the sum of $s$ relative hyperplanes.
\item $b=0$. Then $\rho(X)=1$, $N$ is ample, and $-(K_X+\Delta)\equiv mN$; in particular $-(K_X+\Delta)$ is ample.
\end{enumerate}
\end{theorem}

\begin{proof}
Lemma~\ref{lem:conformal-degree} gives $n=2m$, and Lemma~\ref{lem:delta1-covering}\textup{(ii)} shows that $W$ covers $X$. Since every $W$-curve is contracted by $\phi_R$, the contraction is of fiber type. If $b=0$, then $\rho(X)=1$ and Lemma~\ref{lem:det} gives $-(K_X+\Delta)\equiv mN$; since $N$ is positive on the unique extremal ray, it is ample.

Assume $b>0$ and let $F$ be a general fiber. After shrinking the base, $F$ is smooth and $D_F:=\Delta|_F$ is simple normal crossing; vertical components of $\Delta$ do not meet $F$. Since $\rho(X/Y)=1$ and $N\cdot R=1$, the line bundle $N$ is $\phi_R$-ample; hence $N_F:=N|_F$ is ample. With $E:=T_X(-\log\Delta)$, the relative logarithmic tangent sequence is $0\to T_F(-\log D_F)\to E|_F\to\OO_F^{\oplus b}\to0.$
Dualising the quotient $E|_F\twoheadrightarrow\OO_F^{\oplus b}$ gives an injection $\OO_F^{\oplus b}\hookrightarrow E^\vee|_F$. Tensoring by $N_F$ and using $E|_F\simeq E^\vee|_F\otimes N_F$ yields $N_F^{\oplus b}\hookrightarrow E|_F.$
The composition with $E|_F\twoheadrightarrow\OO_F^{\oplus b}$ is a morphism $N_F^{\oplus b}\to\OO_F^{\oplus b}$. Since $N_F$ is ample, $H^0(F,N_F^{-1})=0$, so this morphism is zero. Hence the injection factors through the relative logarithmic tangent bundle: $N_F^{\oplus b}\hookrightarrow T_F(-\log D_F).$
A general point of $F$ lies on a $W$-curve, whose unique boundary point lies on the horizontal boundary; thus $D_F\ne0$. Langer's logarithmic Wahl theorem \cite[Theorem~1.1]{Langer24} yields $(F,D_F,N_F)\simeq(\PP^d,H_1+\cdots+H_s,\OO_{\PP^d}(1)),\qquad s\ge1.$
Restricting Lemma~\ref{lem:det} to $F$ gives $-(K_F+D_F)\equiv mN_F$. On $\PP^d$ with $s$ hyperplanes this reads $(d+1-s)H=mH$, hence $d=m+s-1$. Since $\dim X=2m$ and $\dim F=d$, the base has dimension $b=2m-d=m-s+1.$
The assumption $b>0$ is therefore equivalent to $1\le s\le m$.

A $W$-curve in $F$ has $N_F$-degree one, hence is a line. Its reduced inverse image of $D_F$ is one point, so it meets all $H_i$ there and $D_F\cdot C=s$. By \cite[Remark~1.2]{Langer24},
$T_F(-\log D_F)\simeq\OO_F^{\oplus(s-1)}\oplus\OO_F(1)^{\oplus m}$; restricting the relative logarithmic tangent sequence to $C$ and using $H^1(C,T_F(-\log D_F)|_C)=0$ gives
$E|_C\simeq\OO(1)^{\oplus m}\oplus\OO^{\oplus m}$. The ordinary tangent sequence gives the second splitting.

Finally, shrink $Y$ so that the fibers have the form just obtained. Since $H^i(\PP^d,\OO(1))=0$ for $i>0$ and $h^0(\PP^d,\OO(1))=d+1$ is constant, cohomology and base change make $\phi_{R*}N$ locally free and the evaluation $\phi_R^*\phi_{R*}N\to N$ surjective. It defines a morphism from $X$ to the associated projective bundle over $Y$; on every fiber this is the standard isomorphism $\PP^d\to\PP(H^0(\PP^d,\OO(1))^\vee)$. After a further shrinking it is an isomorphism. Each horizontal component of $\Delta$ restricts to a hyperplane, hence is a relative hyperplane.
\end{proof}

\subsection*{The case $\delta=0$}

\begin{lemma}\label{lem:exclude-delta0-d1}
If $\delta=0$, then $\deg(f^*N)=2$.
\end{lemma}

\begin{proof}
Assume $d:=\deg(f^*N)=1$. By Lemma~\ref{lem:conformal-degree}, $n=2m$ and $\ell(R)=m$. Since $\Delta\cdot R=0$, this is also the ordinary Mori length of $R$. Let $V\subset\Hom(\PP^1,X)$ be the component containing $[f]$, set $Z(f):=\locus(V)$ and $F_x(f):=\locus(V,x)$ for a smooth point $x\in f(\PP^1)\cap U$, where $U:=X\setminus\Supp\Delta$. Every member of $V$ meeting $U$ is contained in $U$.

The family $V$ is unsplit by Lemma~\ref{lem:unsplit-minimal}, and Proposition~\ref{prop:Langer-estimates}\textup{(ii)} gives $3m\le\dim V\le\dim Z(f)+\dim F_x(f)+1.$
Set $r:=\dim F_x(f)$. Self-duality on $U$ pairs the splitting degrees of $u^*T_X$ to sum to one, so Proposition~\ref{prop:Langer-estimates}\textup{(ii)} gives $r\le m$, while the preceding inequality gives $r\ge m-1$.

The required information on the splitting follows directly from generic smoothness. Let $v$ be general and write $v^*T_X\simeq\bigoplus_{i=1}^{2m}\OO(a_i)$ with $a_1\ge\cdots\ge a_{2m}$. At a general point $(v,t)$, the differential of the universal evaluation $V\times\PP^1\to Z(f)$ has rank $\dim Z(f)$. Its image is contained in the image of the evaluation map
$$
H^0(\PP^1,v^*T_X)\longrightarrow v^*T_X|_t.
$$
For the displayed splitting, this evaluation map has rank exactly $\#\{i:a_i\ge0\}$. Hence $\#\{i:a_i\ge0\}\ge\dim Z(f).$
If $r=m-1$, then $\dim Z(f)=2m$, so all $a_i$ are nonnegative; but $a_1\ge2$ and self-duality gives $a_{2m}=1-a_1<0$. Hence $r=m$. The same argument excludes $\dim Z(f)=2m$, so $\dim Z(f)=2m-1$ and therefore $a_{2m-1}\ge0$ for a general $v$. The relations $a_i+a_{2m+1-i}=1$ then force
$$
v^*T_X\simeq\OO(a)\oplus\OO(1)^{\oplus(m-1)}\oplus\OO^{\oplus(m-1)}\oplus\OO(1-a),\qquad a\ge2.
$$
Minimality makes $v$ birational onto its image. If $B$ is its ramification divisor, the differential $\OO(2)\to v^*T_X$ can map only to $\OO(a)$, hence $\deg B=a-2$.

Ye's formulas~\textup{(1.5a)} and~\textup{(1.9a)} are used with the following conventions. Let $Y_U:=\PP_{\mathrm{sub}}(T_U)$ be the projective bundle of tangent lines, let $\pi:Y_U\to U$, and let $\mathcal T:=\OO_{Y_U}(-1)\subset\pi^*T_U$ be the tautological subbundle. Set $\mathcal L:=\mathcal T^\vee\otimes\pi^*N$. The conformal isomorphism $T_U\otimes N^{-1}\simeq\Omega_U^1$ identifies this with the weighted tautological line bundle used in \cite{Ye94}. The first formula gives, for the tangent lift $\widetilde v:\PP^1\to Y_U$, $\widetilde v^{\,*}\mathcal L\simeq\Omega_{\PP^1}^1(-B)\otimes v^*N\simeq\OO(1-a).$
Since $v^*N\simeq\OO(1)$, it follows that $\widetilde v^{\,*}\OO_{Y_U}(1)\simeq\OO(-a)$. Pull back the relative Euler sequence $0\to\OO_{Y_U}\to\pi^*T_U\otimes\OO_{Y_U}(1)\to T_{Y_U/U}\to0.$
The first map lands in the trivial summand obtained from $\OO(a)\otimes\OO(-a)$, since the saturated tangent direction of $v$ lies in the summand $\OO(a)$. Consequently
$$
\widetilde v^{\,*}T_{Y_U/U}\simeq\OO(1-2a)\oplus\OO(-a)^{\oplus(m-1)}\oplus\OO(1-a)^{\oplus(m-1)}.
$$
The tangent sequence is $0\to\widetilde v^{\,*}T_{Y_U/U}\to\widetilde v^{\,*}T_{Y_U}\to v^*T_U\to0.$
A line subbundle of $\widetilde v^{\,*}T_{Y_U}$ either maps nontrivially to $v^*T_U$, whose largest summand has degree $a$, or is contained in the relative term, whose largest summand has degree $1-a$. Thus every line subbundle of $\widetilde v^{\,*}T_{Y_U}$ has degree at most $a$.

On the other hand, $v^*g\in H^0\bigl(\PP^1,\Sym^2\Omega_{\PP^1}^1\otimes v^*N\bigr)=H^0(\PP^1,\OO(-3))=0.$
Hence $v$ is null and $\widetilde v(\PP^1)$ lies in the null quadric bundle $S\subset Y_U$. For a ramified null curve, the second formula gives $\widetilde v^{\,*}N_S^0\simeq T_{\PP^1}^{\otimes2}(2B)\otimes v^*N^{-1}\simeq\OO(2a-1),$
where $N_S^0\subset T_S\subset T_{Y_U}|_S$ is the characteristic line. This is a line subbundle of $\widetilde v^{\,*}T_{Y_U}$ of degree $2a-1>a$, a contradiction. For comparison, the formula has the expected normalization on the quadric: for a line $v$ in $Q^n$, $B=0$ and $v^*N\simeq\OO(2)$, so the same formula gives $\widetilde v^{\,*}N_S^0\simeq\OO(2)=T_{\PP^1}$. Hence $d\ne1$, and Lemma~\ref{lem:conformal-degree} gives $d=2$.
\end{proof}

\begin{proposition}\label{prop:quadric-branch}
If $\delta=0$, then $\Delta=\varnothing$ and $X\simeq Q^n$.
\end{proposition}

\begin{proof}
By Lemma~\ref{lem:exclude-delta0-d1}, $\deg(f^*N)=2$, hence $\ell(R)=n$ and $\Delta\cdot R=0$. The ray is $K_X$-negative. The Ionescu--Wi\'sniewski inequality excludes a birational contraction: for a nontrivial fiber component $F$ and an exceptional component $E$ it would give
$\dim E+\dim F\ge2n-1$, whereas both dimensions are at most $n-1$. Thus $\phi_R$ is of fiber type. Taking $E=X$ in the same inequality gives $\dim Y\le1$.

Suppose $\dim Y=1$ and let $F$ be a general fiber. Then $F$ is smooth of dimension $n-1$, $N_F:=N|_F$ is ample, and $0\to T_F(-\log D_F)\to T_X(-\log\Delta)|_F\to\OO_F\to0, \qquad D_F:=\Delta|_F.$
Dualising the quotient $T_X(-\log\Delta)|_F\twoheadrightarrow\OO_F$ and using the conformal self-duality gives an injection $N_F\hookrightarrow T_X(-\log\Delta)|_F$. Its composition with the quotient to $\OO_F$ is zero since $H^0(F,N_F^{-1})=0$, so it factors through $T_F(-\log D_F)$. Langer's logarithmic Wahl theorem gives $F\simeq\PP^{n-1}$ and $N_F\simeq\OO(1)$. A line in $F$ is contracted by $\phi_R$ and has $N$-degree one, contradicting the minimality of $f$, whose $N$-degree is two. Hence $Y$ is a point.

Thus $\rho(X)=1$. Since $\Delta\cdot R=0$, $\Delta\equiv0$, and an effective numerically trivial divisor on a projective variety is zero. Hence $\Delta=\varnothing$. The logarithmic conformal structure is therefore an ordinary holomorphic conformal structure and $K_X$ is not nef; Ye's theorem \cite[Theorem~2.1]{Ye94} gives $X\simeq Q^n$.
\end{proof}

\begin{definition}\label{def:rational-maximal-isotropic-fibration-short}
A \emph{rational maximal isotropic fibration} of a log-conformal pair $(X,\Delta,g)$ of dimension $2m$ is a dominant rational map
$
\pi:X\dashrightarrow Y
$
for which there exist dense open subsets $X^\circ\subset X$ and $Y^\circ\subset Y$ such that $\pi^\circ:X^\circ\to Y^\circ$ is smooth and
$
T_{X^\circ/Y^\circ}(-\log\Delta^\circ)\subset T_{X^\circ}(-\log\Delta^\circ)
$
is a maximally isotropic subbundle with respect to $g|_{X^\circ}$.
\end{definition}

The preceding results give the classification.

\begin{theorem}\label{thm:main}
Let $(X,\Delta)$ be a smooth complex projective simple normal crossing pair of dimension $n\ge 3$ carrying an everywhere nondegenerate log-conformal tensor $(N,g)$. Assume that $K_X+\Delta$ is not nef. Then one of the following mutually exclusive alternatives occurs.
\begin{enumerate}
\item $\Delta=\varnothing$ and $X\simeq Q^n$.
\item $X\simeq\PP^n$ and $\Delta$ is a hyperplane.
\item $n=2m$ is even and there is a $(K_X+\Delta)$-negative elementary contraction $\phi_R:X\to Y$, generated by a minimal rational curve $C$ with $N\cdot C=1$, for which one of the following holds.
\begin{enumerate}
\item $\dim Y>0$. Then, for some $1\le s\le m$,
$
\dim Y=m-s+1,
$
and the geometric generic fiber is
$
(\PP^{m+s-1},H_1+\cdots+H_s),
$
with $H_1,\dots,H_s$ hyperplanes in general position; for a general minimal curve $C$ in this fiber, $s=\Delta\cdot C$. On dense open subsets $X^\circ\subset X$ and $Y^\circ\subset Y$, the restriction $\phi_R^\circ:X^\circ\to Y^\circ$ is a projective bundle, the horizontal part of $\Delta^\circ$ is the sum of $s$ relative hyperplanes, and $N|_{X^\circ}\simeq\OO_{\phi_R^\circ}(1)\otimes(\phi_R^\circ)^*L$
for a line bundle $L$ on $Y^\circ$.
If $s=1$, then $\phi_R^\circ$ is a rational maximal isotropic fibration and there is an exact sequence $0\to\OO_{Y^\circ}\to V\to T_{Y^\circ}\otimes L^{-1}\to0$
such that
$
(X^\circ,\Delta^\circ)\simeq
\bigl(\PP(V),\PP(T_{Y^\circ}\otimes L^{-1})\bigr).
$
\item $Y$ is a point. Then $\rho(X)=1$, $-(K_X+\Delta)\equiv mN$, and $-(K_X+\Delta)$ is ample.
\end{enumerate}
\end{enumerate}
\end{theorem}

\begin{proof}
Choose a $(K_X+\Delta)$-negative extremal ray $R$ and let $f:\PP^1\to X$ be the minimal morphism furnished by Proposition~\ref{prop:A1-normal}. Then $\delta\le1$.

If $\delta=0$, Proposition~\ref{prop:quadric-branch} gives $\Delta=\varnothing$ and $X\simeq Q^n$, which is \textup{(i)}.

Assume $\delta=1$. By Lemma~\ref{lem:conformal-degree}, $\deg(f^*N)\in\{1,2\}$. If $\deg(f^*N)=2$, Theorem~\ref{thm:Pn} gives $X\simeq\PP^n$ and $\Delta$ a hyperplane, which is \textup{(ii)}.

Assume $\deg(f^*N)=1$. Theorem~\ref{thm:third-branch} gives \textup{(iii)(i)} or \textup{(iii)(ii)}. The only additional assertion in Theorem~\ref{thm:main} concerns the case \textup{(iii)(i)} with $s=1$. In this case $\dim Y=m$ and the general fiber is $(\PP^m,H)$.

Choose a dense open subset $Y^\circ\subset Y$ over which the projective-bundle description holds, set $X^\circ:=\phi_R^{-1}(Y^\circ)$, and write $\pi^\circ:=\phi_R|_{X^\circ}:X^\circ\to Y^\circ$. The logarithmic relative tangent bundle
$
P^\circ:=T_{X^\circ/Y^\circ}(-\log\Delta^\circ)
$
has rank $m$. On a line $\ell$ in a fiber, $P^\circ|_\ell\simeq\OO_{\PP^1}(1)^{\oplus m}$ and $N|_\ell\simeq\OO_{\PP^1}(1)$. Hence $H^0\Bigl(\ell,\Sym^2((P^\circ|_\ell)^\vee)\otimes N|_\ell\Bigr)=0.$
Lines cover the fibers, so $g|_{P^\circ\otimes P^\circ}=0$. Hence $P^\circ$ is maximally isotropic.
Since $\Delta^\circ$ is a relative hyperplane, after twisting the vector bundle defining the projective bundle, write
\begin{equation}\label{eq:quotient-normal-form}
0\to\OO_{Y^\circ}\to V\to\cG\to0,
\end{equation}
with $X^\circ\simeq\PP(V)$ and $\Delta^\circ\simeq\PP(\cG)$. This is also the case $s=1$ of \cite[Lemma~5.15]{Langer24}. Moreover
\begin{equation}\label{eq:N-splitting}
N|_{X^\circ}\simeq\OO_{\pi^\circ}(1)\otimes(\pi^\circ)^*L
\end{equation}
for a line bundle $L$ on $Y^\circ$. The relative logarithmic Euler sequence, applied locally where \eqref{eq:quotient-normal-form} splits, identifies
\begin{equation}\label{eq:P-from-G}
P^\circ\simeq\OO_{\pi^\circ}(1)\otimes(\pi^\circ)^*(\cG^\vee).
\end{equation}
Maximal isotropy identifies $T_{X^\circ}(-\log\Delta^\circ)/P^\circ\simeq(P^\circ)^\vee\otimes N|_{X^\circ}.$
Comparing with the relative logarithmic tangent sequence gives
$
(P^\circ)^\vee\otimes N|_{X^\circ}\simeq(\pi^\circ)^*T_{Y^\circ}.
$
Equations~\eqref{eq:N-splitting} and \eqref{eq:P-from-G} give $\cG\simeq T_{Y^\circ}\otimes L^{-1}.$
Substitution in \eqref{eq:quotient-normal-form} gives the required exact sequence and the claimed description of $(X^\circ,\Delta^\circ)$.
\end{proof}

\begin{remark}\label{rem:rho-one-residual}
The Picard-rank-one branch in Theorem~\ref{thm:main}\,\textup{(iii)(ii)} is not empty. Proposition~\ref{prop:multi-hyperplane-example} with $s=m+1$ gives $(X,\Delta,N)=\bigl(\PP^{2m},H_1+\cdots+H_{m+1},\OO_{\PP^{2m}}(1)\bigr).$
The Picard-rank-one case with $-(K_X+\Delta)$ ample is distinct from the positive-dimensional fiber case. When $Y$ is a point, there is no positive-dimensional fiber to which the relative logarithmic tangent sequence can be restricted; the resulting conclusion is therefore $\rho(X)=1$ and $-(K_X+\Delta)\equiv mN$.
\end{remark}

\begin{proposition}[Criterion for the existence of a nondegenerate log-conformal structure]
\label{prop:converse-log-conformal-criterion}
Let $Y$ be a smooth complex variety of dimension $m$, let $L$ be a line bundle on $Y$, and let
\begin{equation}\label{eq:base-extension-criterion}
\begin{tikzcd}[column sep=large]
0 \arrow[r] &
\OO_Y \arrow[r] &
V \arrow[r] &
T_Y \otimes L^{-1} \arrow[r] &
0
\end{tikzcd}
\end{equation}
be an exact sequence of vector bundles. Set $X:=\PP(V), \qquad \Delta:=\PP(T_Y\otimes L^{-1})\subset X, \qquad \pi:X\to Y,$
with the Grothendieck convention $\PP(V)=\Proj(\Sym V)$, and let
$
N:=\OO_\pi(1)\otimes \pi^*L.
$
Set also $P:=T_{X/Y}(-\log\Delta)$ and $E:=T_X(-\log\Delta)$. Then $\Delta$ is a relative hyperplane divisor and the following hold:
\begin{enumerate}
\item There is a canonical isomorphism
\begin{equation}\label{eq:P-identification-criterion}
P \simeq \pi^*\Omega_Y^1 \otimes N.
\end{equation}

\item The relative logarithmic tangent sequence
\begin{equation}\label{eq:relative-log-sequence-criterion}
\begin{tikzcd}[column sep=large]
0 \arrow[r] &
P \arrow[r] &
E \arrow[r] &
\pi^*T_Y \arrow[r] &
0
\end{tikzcd}
\end{equation}
defines a class $\xi(E)\in \Ext^1_X(\pi^*T_Y,P) \simeq H^1\bigl(X,\mathcal{H}om(\pi^*T_Y,P)\bigr).$
Equation~\eqref{eq:P-identification-criterion} identifies
$$
\mathcal{H}om(\pi^*T_Y,P)
  \simeq \pi^*(\Omega_Y^1\otimes \Omega_Y^1)\otimes N
  \simeq
  \bigl(\pi^*\Sym^2\Omega_Y^1\otimes N\bigr)
  \oplus
  \bigl(\pi^*\Lambda^2\Omega_Y^1\otimes N\bigr).
$$

\item There exists a fiberwise nondegenerate tensor $g\in H^0\bigl(X,\Sym^2\Omega_X^1(\log\Delta)\otimes N\bigr)$
such that $P$ is maximally isotropic with respect to $g$ if and only if the component of $\xi(E)$ in $H^1\bigl(X,\pi^*\Sym^2\Omega_Y^1\otimes N\bigr)$
vanishes. That is,
\begin{equation}\label{eq:alternating-condition}
\xi(E)\in H^1\bigl(X,\pi^*\wedge^2\Omega_Y^1\otimes N\bigr)
\subset
H^1\bigl(X,\pi^*(\Omega_Y^1\otimes\Omega_Y^1)\otimes N\bigr).
\end{equation}
\end{enumerate}

In this situation, $P$ has rank $m=\frac12\rk E$, hence it is maximally isotropic whenever such a tensor $g$ exists.
\end{proposition}

\begin{proof}
Since $\Delta$ is a relative hyperplane divisor in the projective bundle $\pi:X=\PP(V)\to Y$, the relative logarithmic Euler sequence can be computed locally on $Y$, where the extension \eqref{eq:base-extension-criterion} splits. It gives $\Omega^1_{X/Y}(\log\Delta)\simeq \pi^*(T_Y\otimes L^{-1})(-1).$
That is, this is the $s=1$ relative hyperplane calculation underlying \cite[Lemmas~5.15--5.17]{Langer24}. Since $N=\OO_\pi(1)\otimes \pi^*L$, dualising yields
$$
P=T_{X/Y}(-\log\Delta)
 \simeq \OO_\pi(1)\otimes \pi^*(\Omega_Y^1\otimes L)
 \simeq \pi^*\Omega_Y^1\otimes N.
$$
This proves \eqref{eq:P-identification-criterion}.
The exact sequence \eqref{eq:relative-log-sequence-criterion} is the standard relative logarithmic tangent sequence for the smooth morphism $\pi$ with relative simple normal crossing divisor $\Delta$. Its extension class $\xi(E)$ therefore lies in $\Ext^1_X(\pi^*T_Y,P)\simeq H^1\bigl(X,\mathcal{H}om(\pi^*T_Y,P)\bigr).$
The identification~\eqref{eq:P-identification-criterion} gives
$$
\mathcal{H}om(\pi^*T_Y,P)
 \simeq (\pi^*T_Y)^\vee\otimes P
 \simeq \pi^*\Omega_Y^1\otimes \pi^*\Omega_Y^1\otimes N
 \simeq \pi^*(\Omega_Y^1\otimes\Omega_Y^1)\otimes N.
$$
In characteristic $0$, there is the canonical decomposition $\Omega_Y^1\otimes\Omega_Y^1 \simeq \Sym^2\Omega_Y^1 \oplus \wedge^2\Omega_Y^1,$
and hence the stated decomposition of $\mathcal{H}om(\pi^*T_Y,P)$.
For \textup{(iii)}, assume first that there exists a tensor
$
g\in H^0\bigl(X,\Sym^2\Omega_X^1(\log\Delta)\otimes N\bigr)
$
which is fiberwise nondegenerate and for which $P$ is isotropic. Since $\rk P = \rk T_{X/Y}(-\log\Delta)=m \qquad\text{and}\qquad \rk E = \rk T_X(-\log\Delta)=2m,$
the isotropic subbundle $P$ is automatically maximal.
Since $g$ is nondegenerate, contraction with $g$ yields an isomorphism
$
g^\sharp:E\xrightarrow{\sim} E^\vee\otimes N.
$
Since $P$ is isotropic, the restriction of $g^\sharp$ to $P$ vanishes on $P$, hence it factors through a morphism $\bar g^\sharp:E/P \to P^\vee\otimes N.$
The induced map $\bar g^\sharp$ is an isomorphism. Indeed, if a class $\bar e\in E/P$ maps to $0$, then any lift $e\in E$ satisfies $g(e,P)=0$, so the subbundle generated by $P$ and $e$ is isotropic. By maximality of $P$, this implies $e\in P$, hence $\bar e=0$. Since both bundles have rank $m$, $\bar g^\sharp$ is an isomorphism. The quotient map in \eqref{eq:relative-log-sequence-criterion} therefore identifies
\begin{equation}\label{eq:quotient-identification-proof}
E/P \simeq P^\vee\otimes N \simeq \pi^*T_Y.
\end{equation}
Choose an open cover $\{U_i\}$ of $X$ such that \eqref{eq:relative-log-sequence-criterion} splits on each $U_i$:
$
E|_{U_i}\simeq P|_{U_i}\oplus \pi^*T_Y|_{U_i}.
$
Via \eqref{eq:quotient-identification-proof}, in each such splitting the form $g$ takes the matrix form
$$
\begin{pmatrix}
0 & I\\
I & 0
\end{pmatrix}
$$
with respect to the decomposition $P\oplus \pi^*T_Y$. Indeed, $g|_{P\otimes P}=0$ by isotropy, $g$ identifies the quotient with $P^\vee\otimes N$, and symmetry forces exactly the hyperbolic form $g_i((\alpha,u),(\beta,v)):=\alpha(v)+\beta(u),$
where $\alpha,\beta\in P|_{U_i}\simeq (\pi^*T_Y|_{U_i})^\vee\otimes N$ and $u,v\in \pi^*T_Y|_{U_i}$.
On overlaps $U_{ij}:=U_i\cap U_j$, the two splittings differ by a unique morphism
$B_{ij}:\pi^*T_Y|_{U_{ij}}\to P|_{U_{ij}}$.
Thus the transition from the $i$-th splitting to the $j$-th splitting is
$
(\alpha,u)\longmapsto (\alpha+B_{ij}(u),u).
$
A direct computation gives $g_j((\alpha,u),(\beta,v)) = g_i((\alpha,u),(\beta,v)) + B_{ij}(u)(v)+B_{ij}(v)(u).$
Since both local expressions come from the same global form $g$, they must coincide, and therefore $B_{ij}(u)(v)+B_{ij}(v)(u)=0 \qquad\text{for all }u,v.$
Hence each $B_{ij}$ is alternating; that is, it lies in
$
\wedge^2(\pi^*\Omega_Y^1)\otimes N \subset \mathcal{H}om(\pi^*T_Y,P).
$
Hence the Čech representative of $\xi(E)$ has values in the alternating summand, proving \eqref{eq:alternating-condition}.
Conversely, assume that the extension class $\xi(E)$ lies in $H^1\bigl(X,\pi^*\wedge^2\Omega_Y^1\otimes N\bigr).$
Choose a cover $\{U_i\}$ and splittings
$
E|_{U_i}\simeq P|_{U_i}\oplus \pi^*T_Y|_{U_i}
$
whose transition maps are determined by a Čech $1$-cocycle
$$
B_{ij}\in \Gamma\bigl(U_{ij},\wedge^2(\pi^*\Omega_Y^1)\otimes N\bigr)
\subset \Gamma\bigl(U_{ij},\mathcal{H}om(\pi^*T_Y,P)\bigr).
$$
On each $U_i$, define an $N$-valued symmetric bilinear form by $g_i((\alpha,u),(\beta,v)):=\alpha(v)+\beta(u).$
This form is evidently symmetric and nondegenerate, and $P|_{U_i}$ is isotropic for $g_i$.
The local forms $g_i$ glue. On $U_{ij}$ the change of splitting is again $(\alpha,u)\longmapsto (\alpha+B_{ij}(u),u).$
Since $B_{ij}$ is alternating,
$
B_{ij}(u)(v)+B_{ij}(v)(u)=0.
$
The preceding transition formula and the alternating property imply that
$g_j=g_i$ on $U_{ij}$.
Hence the $g_i$ define a global section $g\in H^0\bigl(X,\Sym^2\Omega_X^1(\log\Delta)\otimes N\bigr).$
Since nondegeneracy is local and each $g_i$ is nondegenerate, $g$ is fiberwise nondegenerate. Moreover, $P$ is isotropic by construction and hence maximally isotropic, since $\rk P=m=\frac12\rk E$.
\end{proof}

\section{On the numerically trivial case: semi-abelian uniformization}\label{sec:trivial}

Throughout this section assume that $K_X+\Delta\equiv0$ and set $M:=X\setminus\Delta$. Under the metric hypotheses below, a log-conformal tensor forces the restricted holonomy of a complete Ricci-flat K{\"a}hler metric on $M$ to be trivial. Trivial monodromy of the resulting flat connection then yields a trivial logarithmic tangent bundle and a semi-abelian uniformization.

Recent work of Collins and Guenancia \cite{CollinsGuenancia25} shows that, for log Calabi--Yau complements endowed with complete Ricci-flat K{\"a}hler metrics, Bochner-type phenomena and holonomy properties may depend delicately on the geometry of the boundary. The following hypotheses are additional assumptions; in particular, Bochner extension is not deduced here from $K_X+\Delta\equiv0$.

Throughout, $(X,\Delta)$ is a smooth projective simple normal crossing pair of dimension $n\ge 2$.

\begin{definition}\label{def:Bochner-holonomy}
Let $\omega$ be a complete Ricci-flat K{\"a}hler metric on $M$. We say that $(M,\omega)$ satisfies:
\begin{enumerate}
\item[\hypertarget{Bochner-extension}{\textup{(B)}}]
\emph{Bochner extension} if for every $p,q\ge 0$ and every logarithmic tensor bundle $\mathcal E_{p,q}:=T_X(-\log \Delta)^{\otimes p}\otimes \Omega_X^1(\log \Delta)^{\otimes q},$
restriction induces a bijection between $H^0(X,\mathcal E_{p,q})$ and the space of $\nabla^\omega$-parallel sections of $\mathcal E_{p,q}|_M$;

\item[\hypertarget{holonomy-condition}{\textup{(H)}}]
\emph{restricted-holonomy condition} if $\Hol^\circ(\omega)$ is trivial or its action on $T_{M,x}$ is irreducible for some, hence any, point $x\in M$.
\end{enumerate}
\end{definition}

\begin{definition}\label{def:toroidal}
A \emph{toroidal compactification} of a semi-abelian variety $G$ is a smooth projective simple normal crossing pair $(X,\Delta)$ endowed with an algebraic action of $G$ on $X$ whose open orbit is $X\setminus \Delta$.
\end{definition}

\begin{proposition}\label{prop:restricted-holonomy-flat}
Assume $K_X+\Delta\equiv 0$ and that $M$ carries a complete Ricci-flat K{\"a}hler metric $\omega$ satisfying \textup{(B)} and \textup{(H)}. Assume moreover that $(X,\Delta)$ carries a log-conformal tensor $(N,g)$ with $N\simeq \OO_X$. Then $\Hol^\circ(\omega)=\{1\}$. In particular, $\omega$ is flat.
\end{proposition}

\begin{proof}
Since $N\simeq \OO_X$, the log-conformal tensor gives a nonzero section $g\in H^0\bigl(X,\Sym^2\Omega_X^1(\log \Delta)\bigr)$. By \textup{(B)}, its restriction $g|_M$ is $\nabla^\omega$-parallel.
Fix a point $x\in M$, and write $V:=T_{M,x}$. Since $g|_M$ is parallel, the tensor $g_x\in \Sym^2(V^\vee)$ is invariant under $\Hol^\circ(\omega)$. If $\Hol^\circ(\omega)$ is trivial, there is nothing to prove. Otherwise, condition \textup{(H)} implies that its action on $V$ is irreducible. By Berger--Simons theory, the restricted holonomy of a nonflat irreducible Ricci-flat K{\"a}hler metric is $\mathrm{SU}(n)$ or, if $n$ is even, $\mathrm{Sp}(n/2)$, \cite{Besse87,Salamon89,Joyce07}. In the standard representation, neither group preserves a nonzero symmetric complex bilinear form: $\mathrm{SU}(n)$ has no invariant vector in $\Sym^2(V^\vee)$, while $\mathrm{Sp}(n/2)$ preserves its defining skew-symmetric form rather than a symmetric one. Hence $\Sym^2(V^\vee)^{\Hol^\circ(\omega)}=0$ in the nontrivial irreducible cases. Since $g_x\neq0$, these cases are impossible. Thus $\Hol^\circ(\omega)=\{1\}$, and $\omega$ is flat.
\end{proof}

\begin{theorem}\label{thm:semiabelian}
Assume $K_X+\Delta\equiv 0$ and that $M$ carries a complete Ricci-flat K{\"a}hler metric $\omega$ satisfying \textup{(B)} and \textup{(H)}. Assume moreover that $(X,\Delta)$ carries a log-conformal tensor $(N,g)$ with $N\simeq \OO_X$, and that the flat connection on $T_M$ has trivial monodromy. Then $M$ is semi-abelian and $(X,\Delta)$ is a toroidal compactification of $M$.
\end{theorem}

\begin{proof}
By Proposition~\ref{prop:restricted-holonomy-flat}, the metric $\omega$ is flat. Since the flat connection on $T_M$ has trivial monodromy, the tangent bundle $T_M$ admits a global frame of parallel holomorphic vector fields. By \textup{(B)}, these vector fields extend uniquely across the boundary to global logarithmic vector fields $s_1,\dots,s_n\in H^0\bigl(X,T_X(-\log \Delta)\bigr)$. Their restrictions form a frame of $T_M$, so the wedge $s_1\wedge\cdots\wedge s_n$ is a nonzero global section of $\det T_X(-\log \Delta)\simeq \OO_X\bigl(-(K_X+\Delta)\bigr)$.

Since $K_X+\Delta\equiv 0$, the line bundle $\OO_X(-(K_X+\Delta))$ is numerically trivial. Let $D$ be the zero divisor of $s_1\wedge\cdots\wedge s_n$. Then $D$ is an effective divisor numerically equivalent to zero. This forces $D=0$: indeed, if $A$ is any ample divisor on $X$ and $D\neq 0$, then $A^{n-1}\cdot D>0$, contradicting the numerical triviality of $D$. Hence $s_1\wedge\cdots\wedge s_n$ is nowhere vanishing, and the sections $s_1,\dots,s_n$ trivialize $T_X(-\log \Delta)$. Hence $T_X(-\log \Delta)\simeq \OO_X^{\oplus n}$.

By Winkelmann's criterion \cite[Theorem~1]{Winkelmann04}, triviality of the logarithmic tangent bundle on a smooth projective simple normal crossing compactification is equivalent to the existence of a semi-abelian algebraic group acting on $X$ with open orbit $M$. Thus $M$ is semi-abelian and $(X,\Delta)$ is a toroidal compactification of $M$ in the sense of Definition~\ref{def:toroidal}.
\end{proof}

\section{A logarithmic Cartan-geometric interpretation}\label{subsec:cartan-interpretation}

The non-nef alternatives also admit a Cartan-geometric interpretation. The discussion uses only the logarithmic conformal Cartan structure needed for the quadric, the affine-logarithmic projective model, and the split-null case.

Let $V$ be a complex vector space of dimension $n+2$ endowed with a nondegenerate quadratic form $q$, let
$
G:=SO(V,q),
$
and let $P\subset G$ be the stabiliser of a fixed isotropic line $L_0\subset V$. Then the smooth quadric $Q^n$ is the homogeneous space
\begin{equation}\label{eq:quadric-homogeneous-space}
Q^n\simeq G/P.
\end{equation}
The projective geometry of rational homogeneous varieties in their minimal homogeneous embeddings is discussed in \cite{LandsbergManivel03}. Its Lie algebra admits a $|1|$-grading
$$
\mathfrak g=\mathfrak g_{-1}\oplus \mathfrak g_0\oplus \mathfrak g_1,
\qquad
\mathfrak g_0\simeq \mathfrak{co}(n,\C),
\qquad
\mathfrak p=\mathfrak g_0\oplus \mathfrak g_1,
$$
which is the standard algebraic structure underlying holomorphic conformal geometry and its Cartan-geometric formulation, \cite{Sharpe97,CapSchichl00,CapSlovak03,CapSlovak09}. For the logarithmic extension of Cartan geometry to a simple normal crossing boundary, compare \cite[\S3.1]{BiswasDumitrescuMcKay20}.

\begin{definition}\label{def:log-cartan}
Let $(X,\Delta)$ be a smooth projective simple normal crossing pair. A \emph{logarithmic conformal Cartan geometry} of type $(G,P)$ on $(X,\Delta)$ consists of a holomorphic principal $P$-bundle $\pi:\mathcal G\to X$ together with a $\mathfrak g$-valued logarithmic $1$-form
$$
\omega_{\log}\in H^0\bigl(\mathcal G,\Omega^1_{\mathcal G}(\log \mathcal G_\Delta)\otimes \mathfrak g\bigr),
\qquad
\mathcal G_\Delta:=\pi^{-1}(\Delta),
$$
such that:
\begin{enumerate}
\item for every $u\in \mathcal G$, the induced linear map
$
\omega_{\log,u}:T_{\mathcal G,u}(-\log \mathcal G_\Delta)\xrightarrow{\sim}\mathfrak g
$
is an isomorphism;
\item for every $p\in P$, $R_p^*\omega_{\log}=\operatorname{Ad}(p^{-1})\omega_{\log}$;
\item if $A\in \mathfrak p$ and $\zeta_A$ is the fundamental vector field on $\mathcal G$ generated by $A$, then
\begin{equation}\label{eq:fundamental-vector-cartan}
\omega_{\log}(\zeta_A)=A.
\end{equation}
\end{enumerate}
An \emph{exact logarithmic Weyl scale} is a reduction of structure group from $P$ to $G_0:=CO(n,\C)$, that is a section of the quotient bundle $\mathcal G/\exp(\mathfrak g_1)\to X$, whose restriction to $X\setminus \Delta$ is an ordinary exact Weyl structure in the sense of parabolic conformal geometry, \cite{CapSlovak03}, \cite[Chapter~5]{CapSlovak09}. Compare also the general Cartan-geometric background in \cite[Chapter~2]{Sharpe97} and the logarithmic Cartan-geometric setup in \cite[\S3.1]{BiswasDumitrescuMcKay20}.
\end{definition}

On the open part $U:=X\setminus \Delta$, an exact Weyl scale decomposes the Cartan connection into the familiar three pieces $\sigma^*\omega=\theta+\gamma+\rho,$
where $\theta$ is the soldering form with values in $\mathfrak g_{-1}$, $\gamma$ is the Weyl connection with values in $\mathfrak g_0$, and $\rho$ is the Rho-tensor with values in $\mathfrak g_1$; compare \cite[Chapter~5]{CapSlovak09}. In the logarithmic setting one asks that these pieces extend across $\Delta$ with at most logarithmic poles.

\begin{proposition}\label{prop:cartan-implies-tensor}
Assume that $(X,\Delta)$ carries a logarithmic conformal Cartan geometry $(\mathcal G,\omega_{\log})$ of type $(G,P)$. Then there is a canonically induced log-conformal tensor on $(X,\Delta)$. More precisely, if $L:=\mathcal G\times_P L_0, \qquad N:=L^{-2},$
then there is a canonical isomorphism
\begin{equation}\label{eq:tangent-associated-bundle}
T_X(-\log \Delta)\simeq \mathcal G\times_P(\mathfrak g/\mathfrak p)
\end{equation}
and a canonically induced section
$
g\in H^0\bigl(X,\Sym^2\Omega_X^1(\log \Delta)\otimes N\bigr)
$
which is fiberwise nondegenerate.
\end{proposition}

\begin{proof}
\smallskip
\noindent\emph{Step 1: the logarithmic tangent exact sequence.}
Let $x\in X$. Since $\pi:\mathcal G\to X$ is a holomorphic principal $P$-bundle, it is locally analytically trivial near $x$. Thus, after shrinking around $x$, write $(\mathcal G,\mathcal G_\Delta)\simeq (U\times P,(\Delta\cap U)\times P).$
Choose coordinates $(x_1,\dots,x_n)$ on $U$ such that
$
\Delta\cap U=\{x_1\cdots x_r=0\}
$
for some $0\le r\le n$, and coordinates $(y_1,\dots,y_m)$ on $P$. By the standard local description of logarithmic vector fields for an snc divisor \cite[Chapter~IV, \S1.2]{Ogus18}, together with the principal-bundle/logarithmic-Atiyah background discussed in \cite[\S2]{BiswasDumitrescuMcKay20}, the sheaf $T_{U\times P}(-\log((\Delta\cap U)\times P))$ is freely generated by
$$
x_1\frac{\partial}{\partial x_1},\dots,x_r\frac{\partial}{\partial x_r},
\frac{\partial}{\partial x_{r+1}},\dots,\frac{\partial}{\partial x_n},
\frac{\partial}{\partial y_1},\dots,\frac{\partial}{\partial y_m}.
$$
The first $n$ generators project under $d\pi$ onto the standard local generators of $T_X(-\log \Delta)$, while the last $m$ generators span the relative tangent bundle $T_{\mathcal G/X}$ and are killed by $d\pi$. Hence the local coordinate description shows that $d\pi$ restricts to a surjective morphism
\begin{equation}\label{eq:local-log-tangent-sequence}
0\longrightarrow T_{\mathcal G/X}\longrightarrow T_{\mathcal G}(-\log \mathcal G_\Delta)\xrightarrow{\,d\pi\,} \pi^*T_X(-\log \Delta)\longrightarrow 0,
\end{equation}
whose kernel is precisely $T_{\mathcal G/X}$. Since exactness is local on the base, this yields the claimed global exact sequence of locally free sheaves.

\smallskip
\noindent\emph{Step 2: quotienting the Cartan form by $\mathfrak p$.}
At every point $u\in \mathcal G$, the logarithmic Cartan form gives an isomorphism
 $
\omega_{\log,u}:T_{\mathcal G,u}(-\log \mathcal G_\Delta)\xrightarrow{\sim}\mathfrak g.
$
By \eqref{eq:fundamental-vector-cartan}, the subspace $T_{\mathcal G/X,u}$ is mapped isomorphically onto $\mathfrak p$. Hence the composition with the quotient map $\mathfrak g\to \mathfrak g/\mathfrak p$ vanishes precisely on $T_{\mathcal G/X}$. Using \eqref{eq:local-log-tangent-sequence}, it descends to an isomorphism $\pi^*T_X(-\log \Delta)\xrightarrow{\sim}\mathcal G\times (\mathfrak g/\mathfrak p).$
The $P$-equivariance of $\omega_{\log}$ implies that this descended map is $P$-equivariant. Passing to associated bundles yields \eqref{eq:tangent-associated-bundle}.

\smallskip
\noindent\emph{Step 3: the quadratic form on $\mathfrak g/\mathfrak p$.}
The tangent space to the quadric $Q^n=G/P$ at the base point $[L_0]$ identifies with $T_{[L_0]}Q^n\simeq \operatorname{Hom}(L_0,L_0^{\perp}/L_0),$
and the isotropy representation of $P$ on $\mathfrak g/\mathfrak p$ is the natural one on this tangent space, \cite[\S4.1--\S4.2]{CapSchichl00}, \cite[Chapter~4]{CapSlovak09}. Since $L_0$ is isotropic, the ambient quadratic form $q$ induces a nondegenerate quadratic form $\overline q\in \Sym^2\bigl((L_0^{\perp}/L_0)^\vee\bigr).$
Given $\varphi,\psi\in \operatorname{Hom}(L_0,L_0^{\perp}/L_0)$ and $\ell\in L_0$, we define
$
B(\varphi,\psi)(\ell\otimes \ell):=\overline q\bigl(\varphi(\ell),\psi(\ell)\bigr).
$
This yields a $P$-equivariant, fiberwise nondegenerate symmetric bilinear map
\begin{equation}\label{eq:bilinear-pairing-gp}
B:\Sym^2(\mathfrak g/\mathfrak p)\longrightarrow (L_0^\vee)^{\otimes 2}.
\end{equation}
By construction, the associated weight line is precisely $N=L^{-2}$.

\smallskip
\noindent\emph{Step 4: passage to the associated bundle on $X$.}
Using \eqref{eq:tangent-associated-bundle}, the pairing \eqref{eq:bilinear-pairing-gp} descends to a global section
$$
g\in H^0\bigl(X,\Sym^2(T_X(-\log \Delta)^\vee)\otimes N\bigr)
=H^0\bigl(X,\Sym^2\Omega_X^1(\log \Delta)\otimes N\bigr).
$$
Its fiberwise nondegeneracy follows from the nondegeneracy of $\overline q$, so it defines the required logarithmic conformal tensor.
\end{proof}

\begin{remark}\label{rem:cartan-not-converse}
Proposition~\ref{prop:cartan-implies-tensor} proves the implication $\text{logarithmic conformal Cartan geometry} \Longrightarrow \text{log-conformal tensor}.$
Conversely, a general tensor
 $
g\in H^0\bigl(X,\Sym^2\Omega_X^1(\log \Delta)\otimes N\bigr)
 $
does not by itself supply the additional data of a logarithmic Cartan connection. The explicit cases of the classification nevertheless admit natural Cartan-geometric descriptions.
\end{remark}

\begin{remark}\label{rem:three-branches-cartan}
The three cases of the non-nef classification admit the following Cartan-geometric interpretations.

\smallskip
\noindent\emph{\textup{(i)} The homogeneous flat case.} When $\Delta=\varnothing$ and $X\simeq Q^n$, this is the ordinary flat conformal Cartan geometry on the homogeneous space $G/P$ of \eqref{eq:quadric-homogeneous-space}, Proposition~\ref{prop:quadric-example}, Proposition~\ref{prop:quadric-sharp}. Take $\mathcal G=G\to G/P$ and $\omega_{\log}=\omega_{\mathrm{MC}}$, the Maurer--Cartan form. Proposition~\ref{prop:cartan-implies-tensor} then yields exactly the quadric tensor described explicitly in Section~\ref{sec:setup}.

\smallskip
\noindent\emph{\textup{(ii)} The affine-logarithmic projective case.} Proposition~\ref{prop:projective-example} gives the logarithmic tensor on $(\PP^n,H)$, together with the global logarithmic coframe \eqref{eq:omegai-global}, the local boundary expression \eqref{eq:gH-local}, and the tangent-theoretic interpretation recorded in Remark~\ref{rem:projective-tangent-frame}. On the affine chart $\PP^n\setminus H\simeq \A^n$, Remark~\ref{rem:projective-local} shows that this tensor is simply the constant flat quadratic form. Thus $(\PP^n,H)$ is the logarithmic compactification of the flat affine chart in an exact Weyl gauge, rather than a second compact homogeneous model; compare also the affine-logarithmic discussion in \cite[\S4]{BiswasDumitrescuMcKay20}.

\smallskip
\noindent\emph{\textup{(iii)} The split-null case.} The third alternative in Theorem~\ref{thm:main} is the even-dimensional case described geometrically in Theorem~\ref{thm:third-branch}. Proposition~\ref{prop:abelian-projective-example} gives its $s=1$ model and Proposition~\ref{prop:multi-hyperplane-example} gives the multi-hyperplane models. In the $s=1$ subcase, to which the following Cartan-geometric description applies verbatim, the relevant decomposition is of Witt type: $T_X(-\log \Delta)=A\oplus B, \qquad \rk A=\rk B=m, \qquad g|_A=0, \qquad g|_B=0.$
The mixed part of $g$ then induces a perfect pairing
\begin{equation}\label{eq:split-null-perfect-pairing-final}
A\otimes B\longrightarrow N,
\end{equation}
or, that is, an isomorphism
$$
A\xrightarrow{\sim} B^\vee\otimes N
\
\text{or, symmetrically,}
\
B\xrightarrow{\sim} A^\vee\otimes N.
$$
The splitting is not orthogonal: each summand is isotropic, whereas the mixed pairing between the two summands is perfect. Fiberwise, if $u,u'\in A_x$ and $v,v'\in B_x$, then
 $
g_x(u,u')=0$,
 $g_x(v,v')=0$,
$g_x(u,v)$ may be nonzero,
and the induced map $A_x\to B_x^\vee\otimes N_x$ is an isomorphism. Accordingly, in a local frame adapted to the splitting, the matrix of $g$ has the block form
$$
\begin{pmatrix}
0 & M\\
M^t & 0
\end{pmatrix},
$$
with $M$ invertible. After choosing a local trivialization of $N$ and local frames of $A$ and $B$ which are dual with respect to the pairing \eqref{eq:split-null-perfect-pairing-final}, the matrix can be normalised to
$$
\begin{pmatrix}
0 & I_m\\
I_m & 0
\end{pmatrix}.
$$
Thus the nondegeneracy of the tensor arises precisely from the fact that the two isotropic summands are mutually dual. The same feature occurs in the $s=1$ model of Proposition~\ref{prop:abelian-projective-example}; the fibrational interpretation is recorded in Remark~\ref{rem:fibrational-model}. Thus the $s=1$ split-null geometry is an internal reduction of the same conformal structure in even dimension, from $CO(2m,\C)$ to the subgroup preserving an ordered Witt decomposition. For $s>1$ the linewise Witt splitting remains the same, while the relative logarithmic tangent bundle contains additional trivial directions coming from the hyperplane arrangement; the global structure is described by Theorem~\ref{thm:third-branch}.
\end{remark}

\section{Boundary contractions and compact singular models}\label{subsec:boundary-contractions}

The dimension restriction $n\ge3$ from the classification theorem is not imposed here; the final examples are surfaces. The discussion is restricted to birational contractions for which the boundary is contracted to codimension at least two. No divisorial boundary then survives on the target, and the conformal structure is naturally reflexive. If $X$ is a normal variety, we write $T_X$ for its reflexive tangent sheaf. That is, if
$
j:X_{\mathrm{reg}}\hookrightarrow X
$
denotes the inclusion of the smooth locus, then
$
T_X\simeq j_*T_{X_{\mathrm{reg}}}.
$

\begin{proposition}\label{prop:boundary-contraction-codim-two}
Let $(\widetilde{X},\Delta)$ be a smooth projective pair with simple normal crossings boundary, endowed with a logarithmic conformal structure
\begin{equation}\label{eq:boundary-contraction-upstairs-sharp}
g^\sharp:T_{\widetilde{X}}(-\log \Delta)\xrightarrow{\sim}\Omega_{\widetilde{X}}^1(\log \Delta)\otimes N.
\end{equation}
Let
$
\mu:\widetilde{X}\to Y
$
be a birational contraction onto a normal projective variety. Assume that $\mu$ restricts to an isomorphism
$
\widetilde{X}\setminus \Delta \xrightarrow{\sim} Y\setminus \mu(\Delta)
$
and that
$
\operatorname{codim}_Y \mu(\Delta)\geq 2.
$
If the conformal tensor induced by \eqref{eq:boundary-contraction-upstairs-sharp} on the open set
$
U:=\widetilde{X}\setminus \Delta \simeq Y\setminus \mu(\Delta)
$
descends to $Y$, then the descended structure is necessarily reflexive. More precisely, with $j:U\hookrightarrow Y$ and
$\mathcal L:=j_*(N|_U)$, a rank-one reflexive sheaf, there is an isomorphism $q_Y^\sharp:T_Y\xrightarrow{\sim}\bigl(\Omega_Y^{[1]}\otimes\mathcal L\bigr)^{**}.$
\end{proposition}

\begin{proof}
Over $U$, the logarithmic conformal structure is an ordinary conformal structure. Hence \eqref{eq:boundary-contraction-upstairs-sharp} restricts to an isomorphism
$
g^\sharp|_U:T_U\xrightarrow{\sim}\Omega_U^1\otimes N|_U.
$
Since $Y$ is normal and $Y\setminus U$ has codimension at least two, $\mathcal L=j_*(N|_U)$ is rank-one reflexive, and morphisms between reflexive sheaves are determined on $U$. Hence the isomorphism on $U$ extends uniquely to $q_Y^\sharp:T_Y\to\bigl(\Omega_Y^{[1]}\otimes\mathcal L\bigr)^{**}.$
It is an isomorphism on the big open set $U$, hence an isomorphism of reflexive sheaves on $Y$.
\end{proof}

\begin{remark}\label{rem:compact-model-branch-three}
Proposition~\ref{prop:abelian-projective-example} gives the basic smooth compact model corresponding to case~\textup{(iii)} of Theorem~\ref{thm:main}. Indeed, if $A$ is an abelian variety of dimension $r$ and $H\subset \PP^r$ is a hyperplane, then
$
(\widetilde X,\Delta)=\bigl(A\times \PP^r,\;A\times H\bigr)
$
is a smooth projective simple normal crossing pair of dimension $2r$ endowed with an everywhere nondegenerate logarithmic conformal tensor, and the natural projection to $A$ has general fiber $\PP^r$ with hyperplane boundary. Thus this pair furnishes the basic smooth compact model for the third case of the classification theorem.
\end{remark}

\begin{example}\label{rem:generalized-cones-branch-three}
Let $Z$ be a smooth projective variety of dimension $r$, let $\mathcal{M}$ be a very ample line bundle on $Z$, and consider the projective bundle
$
Y:=\mathbb{P}_Z\bigl(\mathcal{M}\oplus \mathcal{O}_Z^{\oplus r}\bigr)
$
with projection
$
\pi:Y\to Z.
$
Let
$
E:=\mathbb{P}_Z\bigl(\mathcal{O}_Z^{\oplus r}\bigr)\subset Y.
$
Then every fiber of $\pi$ is isomorphic to $\PP^r$, and the restriction of $E$ to each fiber is a hyperplane. Thus the pair $(Y,E)$ realizes the fiberwise geometry of case~\textup{(iii)}.

The globally generated line bundle $\OO_Y(1)$ induces a birational contraction
$
e:Y\to X,
$
where $X$ is the corresponding generalized cone, and $e$ contracts $E$ onto a copy of $\PP^{r-1}$. If $\dim Z=r$, then $\dim Y=2r$ and $\operatorname{codim}_X e(E)=r+1\geq 2.$
Hence the boundary disappears completely downstairs. Whenever $(Y,E)$ carries a logarithmic conformal structure compatible with the fibration, Proposition~\ref{prop:boundary-contraction-codim-two} yields a reflexive conformal structure on the singular cone, $q_X^\sharp:T_X\xrightarrow{\sim}\bigl(\Omega_X^{[1]}\otimes\mathcal L\bigr)^{**}.$
Compatibility is restrictive. Twisting the defining bundle by $\mathcal M^{-1}$ gives
$Y\simeq\PP_Z(\OO_Z\oplus\mathcal M^{-1\,\oplus r})$; comparison with the $s=1$ exact sequence in Theorem~\ref{thm:main} forces
$T_Z\simeq(\mathcal M^{-1}\otimes L)^{\oplus r}$ for a line bundle $L$ on $Z$. Thus the construction is conditional. In the natural case $\mathcal M^{-1}\otimes L\simeq\OO_Z$, $T_Z\simeq\OO_Z^{\oplus r}$, so the smooth projective variety $Z$ is an abelian variety; this recovers the basic split-null model of Proposition~\ref{prop:abelian-projective-example}.
\end{example}

\begin{example}\label{ex:surface-quadratic-cover}
Let
$
L_x:=\{x=0\},\,
L_y:=\{y=0\}\subset \PP^2,
$
and set
$
\Delta_2:=L_x+L_y.
$
Then $\Delta_2$ is a reduced simple normal crossing divisor of degree $2$, and
$
K_{\PP^2}+\Delta_2\sim -H,
$
so the pair $(\PP^2,\Delta_2)$ is log non-nef.
Since $\OO_{\PP^2}(\Delta_2)\simeq \OO_{\PP^2}(2)$ is $2$-divisible, there is a natural double cover
$
f_2:S_2\to \PP^2
$
ramified along $\Delta_2$, given in homogeneous coordinates by
$
S_2:=\{w^2=xy\}\subset \PP^3,
$
$
f_2([x:y:z:w])=[x:y:z].
$
The surface $S_2$ is the quadric cone. Its unique singular point is
$
p=[0:0:1:0].
$
Indeed, on the affine chart $z\neq 0$, with coordinates
$
u=x/z,\,
v=y/z,\,
t=w/z,
$
the equation becomes
$
t^2=uv,
$
which is the ordinary double point of type $A_1$.

$$
\begin{tikzcd}[column sep=large,row sep=large]
S_2 \arrow[r,"f_2"] & \PP^2 \\
R_2 \arrow[u,hook] \arrow[r] & \Delta_2 \arrow[u,hook]
\end{tikzcd}
$$

Let
$
\nu_2:\widetilde S_2\to S_2
$
be the minimal resolution. Then
$
\widetilde S_2\simeq \PP_{\PP^1}\bigl(\OO_{\PP^1}\oplus \OO_{\PP^1}(2)\bigr)\simeq \mathbf F_2.
$
Let
$
C_0\subset \widetilde S_2
$
be the negative section, so
$
C_0^2=-2.
$
The reduced ramification divisor on $S_2$ is
$
R_2=R_x+R_y,
$
where
$
R_x=\{x=w=0\},\,
R_y=\{y=w=0\}.
$
Let
$
F_x,\,
F_y\subset \widetilde S_2
$
be their strict transforms. Then $F_x$ and $F_y$ are fibers of the ruling
$
\widetilde S_2\to \PP^1.
$
Set
$
B_2:=C_0+F_x+F_y.
$
This is a reduced simple normal crossing divisor.
The birational picture is therefore more naturally expressed at the level of pairs:
$$
\begin{tikzcd}[column sep=large,row sep=large]
(\widetilde S_2,B_2) \arrow[r,"\nu_2"] \arrow[dr,bend right=15,"\pi"'] & (S_2,R_2) \arrow[d,"f_2"] \\
& (\PP^2,\Delta_2)
\end{tikzcd}
$$
where $\pi:=f_2\circ \nu_2$, the reduced ramification divisor on $S_2$ is
$
R_2=R_x+R_y,
$
and the natural boundary upstairs is
$
B_2=C_0+F_x+F_y.
$
Moreover $C_0$ is exactly the distinguished section contracted by $\nu_2$, while $F_x$ and $F_y$ map birationally onto $R_x$ and $R_y$.

The logarithmic cotangent bundle of $(\widetilde S_2,B_2)$ is computed as follows. Let
$
F
$
denote the numerical class of a fiber, and let
$
C_\infty:=C_0+2F
$
be the positive section. The full toric boundary of $\mathbf F_2$ is
$
B_{\mathrm{tor}}=C_0+C_\infty+F_x+F_y.
$
Since $(\mathbf F_2,B_{\mathrm{tor}})$ is a smooth complete toric pair, its logarithmic cotangent bundle is trivial:
$
\Omega_{\mathbf F_2}^1(\log B_{\mathrm{tor}})\simeq \OO_{\mathbf F_2}^{\oplus 2}.
$
Removing the component $C_\infty$ by the residue sequence gives
\begin{equation}\label{eq:F2-residue}
0\longrightarrow \Omega_{\mathbf F_2}^1(\log B_2)
\longrightarrow \OO_{\mathbf F_2}^{\oplus 2}
\longrightarrow \OO_{C_\infty}
\longrightarrow 0.
\end{equation}
Any nonzero morphism $\OO_{\mathbf F_2}^{\oplus 2}\to \OO_{C_\infty}$ is, after a linear change of basis in $\OO_{\mathbf F_2}^{\oplus 2}$, the projection onto one factor followed by the canonical quotient $\OO_{\mathbf F_2}\to \OO_{C_\infty}$. Hence the kernel in \eqref{eq:F2-residue} is
$
\OO_{\mathbf F_2}\oplus \OO_{\mathbf F_2}(-C_\infty).
$
Therefore
\begin{equation}\label{eq:F2-log-cotangent}
\Omega_{\mathbf F_2}^1(\log B_2)\simeq
\OO_{\mathbf F_2}\oplus \OO_{\mathbf F_2}(-C_\infty),
\qquad
T_{\mathbf F_2}(-\log B_2)\simeq
\OO_{\mathbf F_2}\oplus \OO_{\mathbf F_2}(C_\infty).
\end{equation}

Next,
$
K_{\mathbf F_2}\sim -2C_0-4F
$
and
$
B_2\sim C_0+2F=C_\infty,
$
hence $K_{\mathbf F_2}+B_2\sim -C_\infty.$
In particular, $K_{\mathbf F_2}+B_2$ is not nef. If we set
$
N_2:=\OO_{\mathbf F_2}(C_\infty),
$
then \eqref{eq:F2-log-cotangent} gives
$
\Omega_{\mathbf F_2}^1(\log B_2)\otimes N_2
\simeq
\OO_{\mathbf F_2}(C_\infty)\oplus \OO_{\mathbf F_2},
$
which differs from $T_{\mathbf F_2}(-\log B_2)$ only by permuting the two direct summands. The mixed term therefore defines a fiberwise nondegenerate symmetric bilinear form of Witt type, namely a section
$$
g_2\in H^0\bigl(\mathbf F_2,\Sym^2\Omega_{\mathbf F_2}^1(\log B_2)\otimes \OO_{\mathbf F_2}(C_\infty)\bigr),
$$
whose associated map is an isomorphism
$$
g_2^\sharp:
T_{\mathbf F_2}(-\log B_2)\xrightarrow{\sim}
\Omega_{\mathbf F_2}^1(\log B_2)\otimes \OO_{\mathbf F_2}(C_\infty).
$$

Thus $(\widetilde S_2,B_2)$ is an explicit log-conformal surface in the log non-nef regime. Moreover, this example is the surface case of the generalized cone construction discussed in Example~\ref{rem:generalized-cones-branch-three}: indeed
$
\widetilde S_2=\PP_{\PP^1}(\OO_{\PP^1}\oplus \OO_{\PP^1}(2))
$
and contracting the distinguished section $C_0$ yields the quadric cone $S_2$.
\end{example}

\begin{example}\label{ex:surface-cubic-cover}
Let
$
L_x:=\{x=0\},\,
L_y:=\{y=0\},\,
L_z:=\{z=0\}\subset \PP^2,
$
and set
$
\Delta_3:=L_x+L_y+L_z.
$
Then $\Delta_3$ is the toric boundary of $\PP^2$, and
$
K_{\PP^2}+\Delta_3\sim 0.
$
Thus $(\PP^2,\Delta_3)$ lies in the log-trivial regime.
Since $\OO_{\PP^2}(\Delta_3)\simeq \OO_{\PP^2}(3)$ is $3$-divisible, there is a natural cyclic triple cover
$
f_3:S_3\to \PP^2
$
ramified along $\Delta_3$, given by
$
S_3:=\{w^3=xyz\}\subset \PP^3,
$
$
f_3([x:y:z:w])=[x:y:z].
$
The surface $S_3$ is a cubic surface with exactly three singular points,
$
[1:0:0:0],\,
[0:1:0:0],\,
[0:0:1:0].
$
For instance, on the affine chart $x\neq 0$, with coordinates
$
u=y/x,\,
v=z/x,\,
t=w/x,
$
the equation becomes
$
t^3=uv,
$
which is the rational double point of type $A_2$. The same computation applies cyclically at the other two singular points.

$$
\begin{tikzcd}[column sep=large,row sep=large]
S_3 \arrow[r,"f_3"] & \PP^2 \\
R_3 \arrow[u,hook] \arrow[r] & \Delta_3 \arrow[u,hook]
\end{tikzcd}
$$

Let
$
\nu_3:\widetilde S_3\to S_3
$
be the minimal resolution. Since the singularities are of type $A_2$, the exceptional divisor over each singular point is a chain of two $(-2)$-curves. Hence the total exceptional divisor consists of six smooth rational curves.
The reduced ramification divisor on $S_3$ is
$
R_3=R_x+R_y+R_z,
$
where
$
R_x=\{x=w=0\},\,
R_y=\{y=w=0\},\,
R_z=\{z=w=0\}.
$
Let
$
\widetilde R_x,\,
\widetilde R_y,\,
\widetilde R_z
$
be the strict transforms of these three curves in $\widetilde S_3$, and let
$
E_{\mathrm{exc}}
$
be the exceptional divisor of $\nu_3$. Set
$
B_3:=\widetilde R_x+\widetilde R_y+\widetilde R_z+E_{\mathrm{exc}}.
$
The birational picture is therefore more naturally expressed at the level of pairs:
$$
\begin{tikzcd}[column sep=large,row sep=large]
(\widetilde S_3,B_3) \arrow[r,"\nu_3"] \arrow[dr,bend right=15,"\pi"'] & (S_3,R_3) \arrow[d,"f_3"] \\
& (\PP^2,\Delta_3)
\end{tikzcd}
$$
where $\pi:=f_3\circ \nu_3$. Here $R_3$ is the reduced ramification divisor on $S_3$, while $B_3$ is obtained by adding to the strict transforms $\widetilde R_x,\widetilde R_y,\widetilde R_z$ the full exceptional divisor of the minimal resolution.

$B_3$ is the full toric boundary of $\widetilde S_3$. Indeed, the surface $S_3$ is toric: the open subset where $xyz\neq 0$ is a two-dimensional torus, and the divisors $R_x,R_y,R_z$ are exactly the three torus-invariant boundary components on the singular toric surface $S_3$. The minimal resolution $\nu_3$ is the toric minimal resolution obtained by subdividing each singular cone of type $A_2$ by two new rays. Hence the six exceptional curves are also torus-invariant, and the complete toric boundary on $\widetilde S_3$ is precisely
$
B_3.
$
Since $(\widetilde S_3,B_3)$ is a smooth complete toric pair, its logarithmic tangent and cotangent bundles are trivial:
$$
\Omega_{\widetilde S_3}^1(\log B_3)\simeq \OO_{\widetilde S_3}^{\oplus 2},
\qquad
T_{\widetilde S_3}(-\log B_3)\simeq \OO_{\widetilde S_3}^{\oplus 2}.
$$
Choosing any nondegenerate symmetric bilinear form on $\C^2$ yields a section
$
g_3\in H^0\bigl(\widetilde S_3,\Sym^2\Omega_{\widetilde S_3}^1(\log B_3)\bigr)
$
whose associated map is an isomorphism
$
g_3^\sharp:
T_{\widetilde S_3}(-\log B_3)\xrightarrow{\sim}
\Omega_{\widetilde S_3}^1(\log B_3).
$
Finally, the toric boundary identity gives
 $
K_{\widetilde S_3}+B_3\sim 0.
$
Thus $(\widetilde S_3,B_3)$ is an explicit log-conformal surface in the log-trivial regime.
Unlike Example~\ref{ex:surface-quadratic-cover}, this cubic example is not a generalized cone in the sense of Example~\ref{rem:generalized-cones-branch-three}. The distinction is geometric: in the surface case, a generalized cone is obtained by contracting a single distinguished section of a ruled surface and therefore has a unique contraction centre, whereas the cubic surface $S_3$ has three singular points of type $A_2$.
\end{example}

\bibliographystyle{amsalpha}
\bibliography{Biblio}

\end{document}